\documentclass[11pt]{amsart}
\usepackage{amssymb, amsmath, amsmath}
\usepackage[backref]{hyperref}
\hypersetup{hypertexnames=false}
\usepackage{mathrsfs}
\usepackage{amscd}  
\usepackage{fullpage}
\usepackage[all]{xy} 
\usepackage{setspace}
\usepackage[dvipsnames]{xcolor}
\usepackage[margin=1.86  cm, headsep = 0.1 cm, footskip = 0.7 cm]{geometry}
\usepackage[english]{babel}
\usepackage{amsfonts}
\usepackage{xurl}
\usepackage[T1]{fontenc}
\usepackage{comment}
\usepackage[shortlabels]{enumitem}
\usepackage{tikz-cd}
\usepackage{quiver}
\usepackage{mathtools}
\usepackage{thmtools, thm-restate} 

\definecolor{darkred}{rgb}{0.5,0,0}
\definecolor{darkgreen}{rgb}{0,0.5,0}
\definecolor{darkblue}{rgb}{0,0,0.5}
\hypersetup{linkcolor=darkgreen,filecolor=darkgreen,urlcolor=darkgreen,citecolor=darkgreen,colorlinks}

\definecolor{c1}{RGB}{82,114,252}

\definecolor{c2}{RGB}{200,150,200}

\definecolor{c3}{RGB}{10,180,250}

\definecolor{c4}{RGB}{100,250,80}

\newcommand{\Z}{\mathbb{Z}}

\newcommand\Q{\mathbb{Q}}

\renewcommand\P{\mathbb{P}} 
\newcommand{\A}{\mathbb{A}} 
\newcommand\G{\mathbb{G}} 
\renewcommand{\H}{\mathrm{H}}
 
\newcommand{\cO}{\mathcal{O}}
\newcommand{\RefinedObs}[1]{\widetilde{Z}_{0, \Omega}{(#1)}}
\newcommand{\RefinedObsDeg}[2]{\widetilde{Z}^{#2}_{0, \Omega}{(#1)}}
\newcommand{\ol}{\overline}
\newcommand{\rec}{\mathrm{rec}}
\newcommand{\refrec}{\widetilde{\mathrm{rec}}}
\newcommand{\et}{\text{\'et}} 
\newcommand{\lr}[1]{\langle #1 \rangle} 
\newcommand{\globcyc}{\mathbf{z}_{n,S}}
\newcommand{\globcycu}{\mathbf{u}_{n,S}} 
\newcommand{\auto}{\varphi}

\DeclareMathOperator{\spec}{Spec}

\DeclareMathOperator{\pic}{Pic} 
\DeclareMathOperator{\NS}{NS} 
 
\DeclareMathOperator{\Aut}{Aut}

\DeclareMathOperator{\divisor}{div} 
 
\DeclareMathOperator{\sym}{Sym} 

\DeclareMathOperator{\Br}{Br}

\DeclareMathOperator{\PGL}{PGL}

\DeclareMathOperator{\rk}{rk}

\DeclareMathOperator{\Bl}{Bl}

\DeclareMathOperator{\im}{im} 

\DeclareMathOperator{\CH}{CH} 
\DeclareMathOperator{\Gal}{Gal}
\DeclareMathOperator{\Cor}{Cor} 
\DeclareMathOperator{\conn}{conn}
\newcommand{\Brnr}{{\Br_{\mathrm{nr}}}}
\DeclareMathOperator{\Fix}{Fix}
\DeclareMathOperator{\Alb}{Alb}

\newtheorem{theorem}{Theorem}[section]
\newtheorem{proposition}[theorem]{Proposition}
\newtheorem{lemma}[theorem]{Lemma}
\newtheorem{corollary}[theorem]{Corollary}
\theoremstyle{definition}
\newtheorem{definition}[theorem]{Definition}

\newtheorem{notation}[theorem]{Notation}

\newtheorem{example}[theorem]{Example}
\newtheorem{examples}[theorem]{Examples}
\theoremstyle{remark}
\newtheorem{remark}[theorem]{Remark}

\numberwithin{theorem}{section}
\numberwithin{equation}{section}

\usepackage{tcolorbox}
\usepackage{tikz-cd}
\usepackage{adjustbox}
\usepackage{marvosym}
\usepackage{ stmaryrd }
\usepackage{cancel}
\usepackage[normalem]{ulem}

\usepackage[OT2,T1]{fontenc}
\DeclareSymbolFont{cyrletters}{OT2}{wncyr}{m}{n}
\DeclareMathSymbol{\Sha}{\mathalpha}{cyrletters}{"58}
\DeclareMathSymbol{\locconstBr}{\mathalpha}{cyrletters}{"42}

\title{Refined obstructions to local-global principles for 0-cycles} 

\author{Francesca Balestrieri}
\address{Francesca Balestrieri\\ The American University of Paris\\
129 Rue de l'Universit\'{e}\\ 75007 Paris, France}
\email{fbalestrieri@aup.edu}

\author{Anouk Greven}
\address{Anouk Greven\\ Mathematisches Institut\\
Georg-August-Universität Göttingen\\
Bunsenstrasse 3-5\\ 37073\\ Göttingen\\ Germany}
\email{anouk.greven@mathematik.uni-goettingen.de}

\author{Rachel Newton}
\address{Rachel Newton \\ Department of Mathematics\\
King's College London\\
Strand\\ London\\ WC2R 2LS\\
UK}
\email{rachel.newton@kcl.ac.uk}

\author{Soumya Sankar}
\address{Soumya Sankar\\ Mathematical Institute\\ Utrecht University\\
Hans Freudenthalgebouw\\ Utrecht 3584CD \\ The Netherlands}
\email{s.sankar@uu.nl}

\author{Katerina Santicola}
\address{Katerina Santicola\\ Department of Mathematics\\
King's College London\\
Strand\\ London\\ WC2R 2LS\\
UK}
\email{katerina.1.santicola@kcl.ac.uk}

\author{Manoy Trip}
\address{Manoy Trip\\ Bernoulli Institute\\ University of Groningen\\ Nijenborgh 9\\9747AG\\ Groningen\\The Netherlands}
\email{m.t.trip@rug.nl}

\keywords{Zero cycles, torsors, Hasse principle, weak approximation, obstructions to local-global principles}
\subjclass[2020]{Primary:
14G12;  	
Secondary:
11G35, 
14G05, 
14C25, 
14K15,  
11G10.  
}

\date{\today}

\begin{document}

\begin{abstract}
We introduce new `refined' obstructions
to local-global principles for 0-cycles on algebraic varieties over number fields.
Assuming finiteness of relevant Tate--Shafarevich groups, we show that the Hasse principle and weak approximation for 0-cycles on generalised Kummer varieties and bielliptic surfaces are controlled by obstructions of this new type.  
As an additional application of our refined obstructions, we answer a question of Zhang about the relationship between the Brauer--Manin and connected descent obstructions for 0-cycles. 
We also show that a Corwin--Schlank style refined obstruction set coincides with the set of global 0-cycles, conditionally on the Section Conjecture.
\end{abstract}

\maketitle


\tableofcontents

\section{Introduction}
\label{sec:intro}

Let $k$ be a number field and let $\Omega$ be its set of places. The Hasse principle for rational points is said to hold in a family of varieties over $k$ if, for any variety $X$ in the family, 
\[X(k)=\emptyset\iff X(k_\Omega)\coloneq\prod_{v\in\Omega} X(k_v)=\emptyset.\]
Weak approximation for rational points is said to hold if $X(k)$ is dense in $X(k_\Omega)$ with respect to the product of the $v$-adic topologies on the sets of local points $X(k_v)$. In families where the Hasse principle or weak approximation can fail, one hopes to find an obstruction that explains these failures. More precisely, one seeks uniformly-defined obstruction sets $X(k_{\Omega})^{\operatorname{obs}}$ (such as Brauer--Manin sets, \'{e}tale-Brauer sets, or descent sets coming from torsors under linear algebraic groups) for each variety $X$ in the family such that
\[X(k)\subset X(k_{\Omega})^{\operatorname{obs}}\subset X(k_\Omega).\]
If, for all varieties $X$ in the family, we have
\[X(k)= \emptyset \iff X(k_{\Omega})^{\operatorname{obs}}=\emptyset,\] 
then any failure of the Hasse principle for rational points in this family is explained by emptiness of $X(k_{\Omega})^{\operatorname{obs}}$. If, for all varieties $X$ in the family, $X(k_\Omega)^{\operatorname{obs}}\subset \ol{X(k)}$, then any failure of weak approximation for rational points in this family is explained by a strict inclusion $\ol{X(k_\Omega)^{\operatorname{obs}}}\subsetneq X(k_\Omega)$. In this situation,  $X(k_{\Omega})^{\operatorname{obs}}$ is an optimal local proxy for $X(k)$ with respect to density questions, as the two sets have the same closure. 

There are numerous results and conjectures about how obstruction sets govern the qualitative arithmetic  behaviour of $X(k)$ for specific families of varieties (e.g. \cite{Manin71, CT-Fibrations-03, Sko-Torsors, SkoK3}).
In full generality, however, the question of which obstruction sets control the Hasse principle and weak approximation for rational points remains wide open (see e.g. \cite{Sko-bielliptic-99, Poonen10, CorwinSchlank}).

The set $Z_0(X)$ of 0-cycles on $X$ over $k$ is a natural generalisation of the set of rational points $X(k)$. As such, the fundamental questions for 0-cycles are analogous to those for rational points: if $Z_0^\delta(X)$ is the set of 0-cycles of degree $\delta\in\Z$, we can ask whether it is empty, or whether it is dense in $Z_{0, \Omega}^\delta (X):=~\prod_{v \in \Omega} Z^{\delta}_{0}(X_{k_v})$. Many of the obstruction sets for rational points admit natural analogues for 0-cycles: in~\cite{CT-Chow-93}, the Brauer--Manin set for rational points was generalised to a Brauer--Manin set for 0-cycles and in~\cite{balestrieri-berg, Linh_2026, Zhang-0cycles-25}, the descent sets associated to linear algebraic groups were also extended to 0-cycles, as were hybrid obstruction sets such as \'{e}tale-Brauer sets.

Now suppose we have a family of varieties together with obstruction sets $Z_{0, \Omega}^\delta(X)^{\operatorname{obs}}$ such that 
\[Z_0^\delta(X)\subset Z_{0, \Omega}^\delta(X)^{\operatorname{obs}}\subset Z_{0, \Omega}^\delta(X).\]
If, for all varieties $X$ in the family, we have 
\[Z^\delta_{0}(X) = \emptyset\iff Z^\delta_{0, \Omega}(X)^{\operatorname{obs}} = \emptyset,  \]
then any failure of the Hasse principle for 0-cycles of degree $\delta$ in this family is explained by emptiness of $Z_{0, \Omega}^\delta(X)^{\operatorname{obs}}$. If, for all varieties $X$ in the family, $Z_{0, \Omega}^\delta(X)^{\operatorname{obs}}\subset \ol{Z_0^\delta(X)}$, then any failure of weak approximation for 0-cycles of degree $\delta$ in this family is explained by a strict inclusion $\ol{Z_{0, \Omega}^\delta(X)^{\operatorname{obs}}}\subsetneq Z_{0, \Omega}^\delta(X)$, where closures are taken with respect to the weak approximation topology (see Section~\ref{subsec:prelim-weak-approximation}). 

In this paper, we introduce a new type of obstruction set for 0-cycles which is potentially smaller than the one introduced in~\cite{balestrieri-berg}, and show that obstructions of this type explain all failures of the Hasse principle and weak approximation for 0-cycles in certain families of varieties. 
Henceforth, let $X$ be a smooth, quasi-projective, geometrically integral variety defined
over $k$. 

\begin{definition}[Simplified version of Definition \ref{defn:refined-obstruction}]\label{defn:introdefrefined}
Let $\operatorname{obs}$ be an obstruction for rational points on varieties over number fields (e.g.\ the Brauer--Manin obstruction, a descent obstruction, etc).
We define the \emph{refined $\operatorname{obs}$-obstruction set for 0-cycles} on $X$ to be the set 
\[ \RefinedObs{X}^{\operatorname{obs}} \coloneq\refrec_{X, \Omega}\left(\bigoplus_{L/k \ \textrm{finite}}\Z[X(L_{\Omega})^{\operatorname{obs}}]\right)
\subset Z_{0,\Omega}(X)\coloneq \bigsqcup_{\delta\in\Z}Z_{0,\Omega}^\delta(X),\]
where $\refrec_{X, \Omega}$ is the recombining map defined in Definition \ref{defn:refinedlocalrecombiningmap} and $\Z[X(L_{\Omega})^{\operatorname{obs}}]$ denotes the free abelian group generated by $X(L_{\Omega})^{\operatorname{obs}}$. Let $\RefinedObsDeg{X}{\delta}^{\operatorname{obs}} \coloneq \RefinedObs{X}^{\operatorname{obs}}\cap Z_{0,\Omega}^\delta(X)$.
For concrete examples, we refer the reader to Examples~\ref{example:obstruction-sets-main}. 
\end{definition}

The advantage of this definition is that it uses points over different field extensions $L/k$, which allows one to
make full use of the theory of obstructions for rational points
to deduce new results for 0-cycles. This was not possible with previously studied obstruction sets, such as those in~\cite{balestrieri-berg}.
We exhibit the power of our new definition by giving some applications which were unattainable using the definitions in~\cite{balestrieri-berg, Zhang-0cycles-25} alone. 

The first application concerns generalised Kummer varieties: let $A$ be an abelian variety of dimension at least two over $k$, let $\auto \in \Aut_k(A)$ be a non-trivial automorphism of finite order such that the set of fixed points of any non-trivial power of \(\auto\) is finite, and let $T \in \H^1(k, \ker[1 - \varphi])$. To this data, we associate a generalised Kummer variety $X$ (see Definition \ref{defn:Kummer}), which admits a rational map $f: Y \dashrightarrow X$ 
where $Y$ is a $k$-torsor under $A$ and $f: f^{-1}(U) \to U$ is a $U$-torsor under $G:= \lr{\auto}$ 
for some open dense $U \subset X$. For the largest such \(U \subset X\) and with $\locconstBr$ denoting locally constant Brauer classes 
(see the definitions in Section \ref{sec:Xobs}), we prove the following result. Here the refined obstruction set $\RefinedObsDeg{U}{\delta}^{f,\locconstBr}$ combines descent along $f$ with obstruction sets defined by locally constant Brauer classes on twists of $Y$ over finite extensions, see \eqref{eq:fBUobs}.

\begin{theorem}[Theorem \ref{thm:hasseprinciplebody}]
\label{thm:main intro-hasse-principle}
Let \(A, X,  U, G, f\) be as above. Assume that for every finite extension \(L/k\) and every \(\sigma \in \H^1(L, G)\), the Tate--Shafarevich group \(\Sha(A^{\sigma}_L)\) is finite. Let $\delta \in \Z$. Then \[ Z_0^\delta(X)=\emptyset \iff {Z_0^{\delta}(U) = \emptyset \iff \RefinedObsDeg{U}{\delta}^{f,\locconstBr}=\emptyset.}\] Thus, any failure of the Hasse principle for 0-cycles of degree $\delta$ on $X$ (or $U$) is explained by emptiness of \(\RefinedObsDeg{U}{\delta}^{f,\locconstBr}\).
\end{theorem}

\begin{remark}
The fact that our refined obstruction sets are defined in terms of rational points over extensions enables us to exploit the quasi-torsor structure of $f$ and use a result of Manin showing that  all failures of the Hasse principle for rational points on $Y_L^\sigma$ (for $L/k$ finite and $\sigma\in \H^1(L,G)$) are explained by locally constant Brauer--Manin obstructions \cite{Manin71}. A similar approach using the definition in~\cite{balestrieri-berg} would soon hit a wall. We would need either a result stating that the locally constant Brauer--Manin obstruction for 0-cycles (defined as in~\cite{CT-Chow-93, Liang-Arithmetic-13}) explains all failures of the Hasse principle for 0-cycles of any given degree on all \(Y_L^\sigma\) (which, as far as the authors are aware, is only known for 0-cycles of degree 1 by \cite[Proposition 4.6.1]{Liang-Arithmetic-13}), or a way to guarantee surjectivity of the restriction maps $\mathrm{res}_{M/L}: \locconstBr(Y_L^\sigma)/\Br_0(Y_L^\sigma) \to \locconstBr((Y_L^\sigma)_M)/\Br_0((Y_L^\sigma)_M)$ for finite extensions $M/L$.
\end{remark}

One of the steps in the proof of Theorem~\ref{thm:main intro-hasse-principle} involves increasing the ranks of the abelian varieties \(A_L^{\sigma}\)
over finite extensions of \(k\) of controlled degree. 
Using a strong version of Bertini's theorem and work of \cite{daveacta}, we prove the following statement on rank jumps for abelian varieties, which may be of independent interest. Note that this theorem needs no assumption on finiteness of Tate--Shafarevich groups.

\begin{restatable}{theorem}{rankjump}
\label{thm:intro-main-rank-jumps}
Let \(A\) be an abelian variety over a number field \(k\). There exists $N\in\Z_{>0}$ such that for all primes $\ell \geq N$, there exist infinitely many field extensions $L/k$ with $[L:k]=\ell$ and $\rk A(L)>\rk A(k)$. 
\end{restatable}

\begin{remark}
The precise statement needed to prove Theorem~\ref{thm:main intro-hasse-principle} is much weaker than Theorem~\ref{thm:intro-main-rank-jumps}: it is sufficient to show that given an abelian variety \(A\) over a number field \(k\), a finite subgroup scheme \(F \subset A\) and an integer \(n\in \Z_{>0}\), we can always find an extension \(L/k\) of degree coprime to $n$ such that \(A(L) \setminus F (L)\) is non-empty. Ensuring that the rank increases in such an extension is one way to show this. For elliptic curves, this follows from results on rank jumps, e.g.\ by \cite{LemkeOliver-Thorne-ranks,daveacta}. For an abelian variety $A$ of arbitrary dimension over $k$, results in~\cite{FreyJarden} show that \(A(\ol{k})\) is infinitely generated but give us no control over the degrees of extensions in which the rank jumps. On the other hand, \cite{BruinNajman-ranks} discusses constraints on rank jumps in prime degree extensions, but does not guarantee any jump. 
\end{remark}

Now we turn our attention to questions of density and weak approximation for 0-cycles. 

\begin{theorem}
\label{thm:main intro-weak-approx} 
Let \(A, X, U, G, f\) be as in Theorem~\ref{thm:main intro-hasse-principle} and assume that for every finite extension \(L/k\) and every \(\sigma \in \H^1(L, G)\), the Tate--Shafarevich group \(\Sha( A^{\sigma}_L)\) is finite. Let $\delta \in \Z$. Then
\[
\ol{Z_{0}^{\delta}(X)} = \ol{Z_0^{\delta}(U)} = \ol{\RefinedObsDeg{U}{\delta}^{f, \Brnr}}
\] 
where the closure inside $Z_{0, \Omega}^\delta(X)$ is with respect to the weak approximation topology (Definition~\ref{def:weakapproxtop}). Thus, any failure of weak approximation for 0-cycles of degree $\delta$ on $X$ is explained by a strict containment \(\ol{\RefinedObsDeg{U}{\delta}^{f, \Brnr}}\subsetneq Z_{0,\Omega}^\delta(X)\).
\end{theorem}

Note that Theorem~\ref{thm:main intro-weak-approx} also implies that any failure of the Hasse principle for 0-cycles of degree $\delta$ on $X$ (or $U$) is explained by emptiness of \(\RefinedObsDeg{U}{\delta}^{f, \Brnr}\). However, since locally constant Brauer classes are unramified (see~\cite[(6.1.4)]{Sansuc81}), we have \(\RefinedObsDeg{U}{\delta}^{f, \Brnr}\subset \RefinedObsDeg{U}{\delta}^{f, \locconstBr}\). Thus, the result of Theorem~\ref{thm:main intro-hasse-principle} is \emph{a priori} stronger than what is implied for the Hasse principle by Theorem~\ref{thm:main intro-weak-approx}. Furthermore, another potential advantage of Theorem~\ref{thm:main intro-hasse-principle} is that, under its assumption on finiteness of Tate--Shafarevich groups, the quotients of the locally constant Brauer groups relevant to the obstruction (namely the groups $\locconstBr(Y_L^\sigma)/\Br_0(Y_L^\sigma)$) are finite (see Remark~\ref{remark:locconstBr finite}). On the other hand, the authors do not know whether
all failures of weak approximation for 0-cycles of degree $\delta$ on a generalised Kummer variety $X$ are explained by a strict containment $\ol{\RefinedObsDeg{U}{\delta}^{f,\locconstBr}}\subsetneq Z_{0,\Omega}^\delta(X)$.

Our next result applies to quotients of torsors under abelian varieties by the free action of a finite group or a reductive group.

\begin{theorem}[Theorem~\ref{thm:quotabvars}]
\label{thm:intro bielliptic friend}
Let \(X\) be a smooth, proper, geometrically integral \(k\)-variety. Let \(G\) be a linear algebraic group and \(f \colon Y \to X\) an $X$-torsor under $G$, with \(Y\) a \(k\)-torsor under an abelian variety \(A\). Assume that for every finite extension \(L/k\) and every \(\sigma \in \H^1(L,G)\), 
\(\Sha(A^{\sigma}_L)\) is finite. Let $\delta \in \Z$.
Then: 
\begin{enumerate}[(i)]
    \item 
    \(Z_0^{\delta}(X) =\emptyset \iff \RefinedObsDeg{X}{\delta}^{f,\Br}=\emptyset\iff \RefinedObsDeg{X}{\delta}^{f,\locconstBr} = \emptyset\); and
    \item \(Z_0^{\delta}(X)\) is dense in \(\RefinedObsDeg{X}{\delta}^{f,\Br}\) with respect to the weak approximation topology.
\end{enumerate}
\end{theorem}

In particular, Theorem~\ref{thm:intro bielliptic friend} applies to bielliptic surfaces, which were the source of the first counterexample to sufficiency of the Brauer--Manin obstruction for rational points 
(see e.g.\ \cite{Sko-bielliptic-99}, \cite{Creutz-Bielliptic17}).

Next, we apply our new tools to the study of equivalences between different types of obstruction sets for 0-cycles, similar to those for rational points. For rational points, it is a well-known result of Harari \cite{Harari02} that $X(\A_k)^{\Br} = X(\A_k)^{\conn}$, where 
\begin{equation}\label{eq:conn desc pts}
X(\A_k)^{\conn} \coloneq \bigcap_{G \textrm{ connected}} \bigcap_{[f: Y \to X] \in \H^1(X, G)} \bigcup_{\sigma \in \H^1(k, G)} f^{\sigma}(Y^\sigma(\A_k))
\end{equation}
is the descent set associated to all $X$-torsors under connected linear algebraic $k$-groups.

In \cite[Remark 5.8]{Zhang-0cycles-25}, Zhang asked whether one can prove a comparison theorem for $0$-cycles similar to Harari's result using the definitions in \cite{balestrieri-berg}. Specifically, if we use the definition of descent sets for 0-cycles in \cite{balestrieri-berg} and let 
\[ Z_{0, \Omega}(X)^{\conn} \coloneq \bigcap_{G \textrm{ connected}} \bigcap_{[f: Y \to X] \in  \H^1(X, G)} Z_{0, \Omega}(X)^f,\]
then Zhang asked whether the equality
\[ \overline{Z_{0, \Omega}(X)^{\conn} } = Z_{0, \Omega}(X)^{\Br}\]
holds under suitable assumptions. Using our refined obstructions, we exploit the known equalities  $X(\A_L)^{\Br} = X(\A_L)^{\conn}$ for all finite extensions $L/k$, together with \cite[Theorem 5.6]{Zhang-0cycles-25}, to prove the chain of inclusions
\[ \RefinedObs{X}^{\Br} = \RefinedObs{X}^{\conn} \subset {Z}_{0, \Omega}(X)^{\conn}  \subset {Z}_{0, \Omega}(X)^{\Br}.\]
Under some reasonable assumptions, following a strategy
used in the proofs of \cite[Theorem 3.2.1]{Liang-Arithmetic-13} and \cite[Theorem 5.6]{Zhang-0cycles-25}, we show that the closure of the refined Brauer--Manin set with respect to the weak approximation topology coincides with the usual Brauer--Manin set for 0-cycles, i.e. $\overline{\RefinedObs{X}^{\Br}} = Z_{0,\Omega}(X)^{\Br}$ (Proposition \ref{prop:wudhiq}). This provides the missing piece needed to answer Zhang's question positively and, putting everything together, we prove the following result.

\begin{restatable}{theorem}{conndense}
\label{thm:conndense}
Let $X$ be a smooth, proper, geometrically integral variety over a number field $k$. Suppose that
\begin{enumerate}[(i)]
\item \label{part:conndense-cond1} $Z_{0, \Omega}(X)^{\Br}$ is open in $Z_{0, \Omega}(X)$ with respect to the weak approximation topology, and 
\item \label{part:conndense-cond2} for every positive integer $d$, there exists a finite extension $k'/k$ such that for any degree $d$ extension $K/k$ which is linearly disjoint from $k'/k$, the restriction map
\[
\Br(X)/\Br_0(X)\ \to\ \Br(X_K)/\Br_0(X_K)
\]
is surjective.
\end{enumerate}
Then  $\overline{Z_{0, \Omega}(X)^{\conn}} = Z_{0, \Omega}(X)^{\Br}$.
\end{restatable}

\begin{remark}
The assumptions in the statement of the above theorem are satisfied, for example, when $X$ is a rationally connected variety or a K3 surface (see Proposition~\ref{prop:notvacuous}).
\end{remark}

Our final result is an analogue for 0-cycles of a result of \cite{CorwinSchlank} 
asserting that a generalised descent obstruction explains all failures of local-global principles for rational points, conditional on the validity of the Section Conjecture. In~\cite{CorwinSchlank}, Corwin and Schlank consider the notion of \emph{very strong approximation} (VSA) with respect to an \emph{obstruction datum} \((\operatorname{obs},S,k)\), where $S \subset \Omega_k$ is a non-empty subset of places: a variety \(V\) satisfies VSA with respect to this datum if \(V(\A_{k,S})^{\operatorname{obs}} = V(k)\), where \(V(\A_{k,S})^{\operatorname{obs}}\) denotes the projection of $V(\A_{k})^{\operatorname{obs}} \subset V(\A_{k})$ to 
$V(\A_{k,S})$, see Notation~\ref{not:11}. Assuming the validity of the Section Conjecture (\cite[Corollary 6.4]{CorwinSchlank}) and taking $S$ to be a non-empty set of finite places of $k$, they show that any $k$-variety admits some finite open cover such that each open set in this cover satisfies VSA with respect to the finite descent obstruction, see~\cite[Corollary 6.6]{CorwinSchlank}. Definition~\ref{defn:introdefrefined} can be generalised significantly by replacing \(X(L_{\Omega})^{\operatorname{obs}}\) with any set \(\mathscr{X}^{(L)}\subset X(L_{\Omega})\) containing \(X_L(L)\)
(Definition \ref{defn:refined-obstruction}), and generalised even further by using similar sets arising from a decomposition of \(X_L\) into subschemes (Definition~\ref{defn:generalised-refined-obstruction}). We prove an effective version of \cite[Corollary 6.6]{CorwinSchlank}. To state this result, recall that for any smooth, quasi-projective, geometrically integral variety $X$ over $k$, the \emph{finite descent set for rational points} is defined as 
\begin{equation}\label{eq:etobstruction}
X(\A_k)^{\et} := \bigcap_{\substack{F \textrm{ finite lin.\ alg.}\\ k\textrm{-group}}} \bigcap_{[f:Y \to X] \in \H^1(X, F)} \bigcup_{\sigma \in \H^1(k,F)} f^{\sigma}(Y^\sigma(\A_k)).
\end{equation}

\begin{theorem}[Proposition~\ref{prop:onecoverforall} and Corollary~\ref{cor:corwin-schlank-effective-refinement}]
\label{cor:corwin-schlank-intro}
Let $X$ be a smooth, quasi-projective, geometrically integral variety over a number field $k$. Fix a non-empty (not necessarily finite) set of finite places of $k$. Assume that the Section Conjecture as in \cite[Conjecture 6.4]{CorwinSchlank} holds for any finite extension $L/k$. Let $X = \bigcup_{i=1}^s U_i$ be a finite affine open cover. Then there exists an effectively computable open cover $X = \bigcup_{i=1}^s \bigcup_{j=1}^{r_i} U_{i,j}$, refining $\{ U_i \}_{i=1}^s$, such that for all finite extensions $L/k$ and all $U_{i,j}$, the base change $(U_{i, j})_L$ satisfies very strong approximation with respect to the obstruction datum $(\et, S_L, L)$. In particular, if we consider the refined obstruction set associated to the collection $\left\{\left\{U_{i,j}(\A_{L, S_L})^{\et}\right\}_{\substack{i=1, \ldots, s\\ j = 1, \ldots, r_i}}\right\}_{L/k}$, that is, 
\[ 
\widetilde{Z}_{0, S}(X)^{\textrm{CS-}(\et, S, k)} \coloneq\refrec_{X, S}\left(\bigoplus_{L/k \ \textrm{finite}} \bigoplus_{i = 1}^{s}\bigoplus_{j=1}^{r_i} \Z[U_{i,j}(\A_{L, S_L})^{\et}]\right) \subset Z_{0, S}(X),
\]
then $Z_0(X) = \widetilde{Z}_{0, S}(X)^{\textrm{CS-}(\et, S, k)}$.  Moreover, for any degree $\delta \in \Z$, we have $Z^{\delta}_0(X) = \widetilde{Z}^{\delta}_{0, S}(X)^{\textrm{CS-}(\et, S, k)}$. 
\end{theorem}

\subsection{Outline of the paper}
Section \ref{sec:prelims} contains preliminaries on obstruction sets for 0-cycles. In Section \ref{sec:refined} we construct refined obstruction sets for 0-cycles, which are \emph{a priori} finer than those in the existing literature. In Section \ref{subsec:prelim-weak-approximation} we recall the definitions of the Hasse principle and weak approximation for 0-cycles, and prove a bridging result between obstructions for rational points over all finite extensions and obstructions for 0-cycles.
In Section \ref{sec:locconst} we introduce various subgroups of the Brauer group consisting of locally constant Brauer classes, and recall some results on the Hasse principle and weak approximation for torsors under abelian varieties. In Section \ref{section:kummer-descent} we define generalised Kummer varieties and the refined obstruction sets appearing in Theorems~\ref{thm:main intro-hasse-principle} and \ref{thm:main intro-weak-approx}. Theorem~\ref{thm:main intro-weak-approx} is proved in Section~\ref{sec:WAkummer}. Section~\ref{subsec:hasse} concerns the proof of Theorem~\ref{thm:main intro-hasse-principle}, noting that this proof relies on a theorem on rank jumps for abelian varieties (Theorem~\ref{thm:intro-main-rank-jumps}), which we prove in Section \ref{sec:rankjumps}. 
In Section \ref{sec:quotients-general} we give an application of our refined obstructions to bielliptic surfaces, proving Theorem~\ref{thm:intro bielliptic friend} and Corollary~\ref{cor:biell}. In Section \ref{sec:huizhang} we apply our new machinery to answer a question of Zhang on weak approximation for varieties. Finally, in Section \ref{sec:corwin-schlank} we generalise work of Corwin and Schlank to 0-cycles, proving that a certain refined obstruction set is the only obstruction to the Hasse principle for 0-cycles on smooth, quasi-projective, geometrically integral varieties over number fields (Theorem~\ref{cor:corwin-schlank-intro}).

\subsection{Notation and conventions}
\label{subsec:intro-notation}

In this article, \(k\) will always denote a number field, and we write $\ol{k}$ for an algebraic closure of $k$. We write $\Omega_k$ for the set of all places of $k$. We often drop the subscript $k$ from the notation when it is clear from the context.
We use \emph{variety} to mean a separated scheme of finite type over a field.
For a scheme $X$ over $k$, we denote by $X_0$ the set of closed points on $X$. For a closed point \(x \in X_0\), we denote by \(k(x)\) its residue field over $k$.  We write $Z_0(X)$ for the set of 0-cycles on $X$ and $Z_0^\delta(X)$ for the subset of 0-cycles of degree $\delta$. We write $X(k_\Omega) = \prod_{v \in \Omega} X(k_v)$ and we call any point $(x_v)_v \in X(k_\Omega)$ an $\Omega$-local point on $X$. 
Note that the set of $\Omega$-local points $ X(k_\Omega)$ may not be equal to the set of adelic points \(X(\A_k)\) if \(X\) is not proper. If \(L\) is an extension of \(k\), we will denote by $X_L:=X\times_{\spec(k)} \spec(L)$ the base change of \(X\) to \(L\).
Throughout the article, \(G\) will denote an algebraic group over \(k\), and we will indicate when we assume it to be finite.
Unless mentioned otherwise, we will work with fppf topology. Further, for the sake of brevity, for a \(k\)-scheme \(V\) we will write \(\H^1(V,G) \coloneqq \H^1(V,G_V)\) for the fppf cohomology pointed set.

\subsection{Acknowledgements}
\label{subsec:intro-acknowledgements}
We thank Brendan Creutz, Tim Dokchitser, Vladimir Dokchitser, Adam Morgan and James Rawson for useful discussions. Anouk Greven was funded by the Deutsche Forschungsgemeinschaft (DFG) – 398436923 (RTG2491).
Rachel Newton was supported by UKRI Future Leaders Fellowship MR/T041609/1, MR/T041609/2 and UKRI1060. Soumya Sankar was supported by the Dutch Research Council (NWO) grant OCENW.XL21.XL21.011. 
This project began at the \emph{WINE5: Women in Numbers Europe 5} workshop in 2025. We are grateful for the hospitality and support of the Faculty of Science, University of Split during the workshop and we would especially like to thank Marcela Hanzer, Borka Jadrijevi\'{c}, P\i nar K\i l\i\c{c}er and Lejla Smajlovi\'{c} for organizing \emph{WINE5} and making this collaboration possible.
For the purpose of open access, a CC BY public copyright license is applied to any Author Accepted Manuscript version arising from this submission.

\section{Preliminaries}
\label{sec:prelims}

Throughout this section, let $k$ be a number field and let $X$ be a smooth, quasi-projective, geometrically integral variety defined over $k$. We write $X_0$ for the set of closed points on $X$. 
We write $X(\A_k)$ for the set of adelic points on $X$, and we let $X(k_\Omega) \coloneq\Pi_{v \in \Omega} X(k_v)$.
Then $X(\mathbb{A}_k) \subset X(k_\Omega)$, and equality holds when $X$ is proper.

\begin{definition}
A \emph{0-cycle $z$ on $X$} is a formal sum
$$z \coloneq\sum_{\substack{x\in X_0
}} n_x x,$$
with $n_x\in\mathbb{Z}$, such that $n_x = 0$ for all but finitely many $x \in X_0$. The degree of $z$ is given by
$$\deg(z) \coloneq\sum_{\substack{x\in X_0
}} n_x [k(x):k].$$

\end{definition}
We write $Z_0(X)$ for the \emph{set of 0-cycles on $X$} and $Z_0^\delta(X)$ for the \emph{subset of 0-cycles of degree $\delta$}. 

\begin{remark} 
The set $Z_0(X)$ is an abelian group under addition. For 0-cycles $z_1,z_2$ on $X$, we have the equality $\deg(z_1+z_2) = \deg(z_1) + \deg(z_2)$. 
\end{remark}

\begin{definition} 
A 0-cycle $z = \sum_{x\in X_0} n_x x$ on $X$ is called \emph{effective} if $n_x\geq 0$  for all $x\in X_0$. We say that $x \in X_0$ is in the \emph{support of $z$} if $n_x \neq 0$.
\end{definition}

\begin{example}
Fix a point $x\in X(k)$. The degree of the 0-cycle on $X$ given by $x$ is equal to 1. More precisely, effective 0-cycles of degree $1$ on $X$ are $k$-rational points on $X$. 
\end{example}

\begin{definition}
We write 
\begin{equation*}
Z^{\delta}_{0,\Omega}(X) \coloneq\prod_{v \in \Omega} Z^{\delta}_0(X_{k_v}) \quad \textrm{and} \quad Z_{0,\Omega}(X) \coloneq\bigsqcup_{\delta \in \Z} Z^{\delta}_{0,\Omega}(X).
\end{equation*}
We call elements of $Z_{0,\Omega}(X)$ \emph{$\Omega$-local 0-cycles} on $X$.
\end{definition}

\begin{remark}
For any degree $\delta$, we have a natural inclusion
\begin{equation*}
Z_0^\delta(X) \subset Z_{0,\Omega}^\delta(X).
\end{equation*}
This leads us to consider obstruction sets in a similar vein to, for example, Brauer--Manin obstructions to the Hasse principle and weak approximation for rational points.
\end{remark}

\subsection{Brauer--Manin obstruction sets for 0-cycles} 
\label{subsection:obstruction-sets}
The Brauer group of \(X\) is defined as \(\Br (X) \coloneq\H^2_{\et}(X, \G_m)\).
For $X$ a smooth, quasi-projective, geometrically integral variety, an element $\alpha \in \Br(X)$ induces an evaluation map on \(T\)-points \(\alpha \colon X(T) \to \Br(T)\) via pullback along \(T \to X\) for any \(k\)-scheme \(T\). We obtain the following commutative diagram in the setting of rational points: 
\begin{center}
\begin{tikzcd}
\label{fig:Brauer-rational-points}
	& {X(k)} && X(\A_k) \\
	0 & {\Br(k)} && {\bigoplus_{v\in\Omega} \Br(k_v)} && {\mathbb{Q}/\mathbb{Z}} & 0, 
	\arrow[hook, from=1-2, to=1-4]
	\arrow["\alpha", from=1-2, to=2-2]
	\arrow["{\prod_{v\in\Omega} \alpha}", from=1-4, to=2-4]
	\arrow[from=2-1, to=2-2]
	\arrow[from=2-2, to=2-4]
	\arrow["{\sum_{v\in\Omega} \text{inv}_v}"', from=2-4, to=2-6]
	\arrow[from=2-6, to=2-7]
\end{tikzcd}
\end{center}
where the sequence in the bottom is a short exact sequence coming from class field theory (the Albert--Brauer--Hasse--Noether Theorem). In \cite{Manin71}, Manin used this to define a pairing, now called the  \emph{Brauer--Manin pairing} and denoted by $\lr{ \cdot, \cdot}_{BM}$, given by
\[
\lr{ \cdot, \cdot}_{BM}: X(\A_k) \times \Br(X) \to \Q/\Z; \quad ((x_v)_v, \alpha ) \mapsto \sum_{v \in \Omega} \mathrm{inv}_v(\alpha(x_v)).
\]
The pairing is well defined because, for $(x_v)_v\in \A_k$, we have \(\mathrm{inv}_v(\alpha(x_v)) = 0\) for all but finitely many \(v\). 

We define the \emph{unramified Brauer group}, denoted by \(\Br_{\mathrm{nr}}(X)\), to be the Brauer group of any smooth compactification \(X^c\) of \(X\), which exists by Hironaka's resolution of singularities theorem. This is well defined, as the Brauer group is a birational invariant of proper, integral, regular varieties, and thus \(\Brnr{X}\) is independent of the choice of compactification (see e.g.\ \cite[Proposition 6.2.7]{CTS-BrauerBook}). For $\alpha\in\Brnr(X)$ and $(x_v)_v\in X(k_\Omega)$, we have \(\mathrm{inv}_v(\alpha(x_v)) = 0\) for all but finitely many \(v\) and hence we obtain a well-defined Brauer--Manin pairing
\[
\lr{ \cdot, \cdot}_{BM}: X(k_\Omega) \times \Brnr(X) \to \Q/\Z; \quad ((x_v)_v, \alpha ) \mapsto \sum_{v \in \Omega} \mathrm{inv}_v(\alpha(x_v)).
\]
Note that if $X$ is proper then $X(k_\Omega)=X(\A_k)$ and $\Br_{\mathrm{nr}}(X) = \Br (X)$.

For $\alpha \in \Br(X)$, we define the set of \emph{adelic points orthogonal to $\alpha$} to be
\[X(\A_k)^{\alpha}\coloneq\{x\in X(\A_k)\;:\; \lr{x,\alpha}_{BM}=0\}\]
and, for $\alpha\in \Brnr(X)$, we define the set of \emph{\(\Omega\)-local points orthogonal to \(\alpha\)} to be
\[X(k_\Omega)^{\alpha}\coloneq\{x\in X(k_\Omega)\;:\; \lr{x,\alpha}_{BM}=0\}.\]
For $B \subset \Br(X)$, we define the set of \emph{adelic points orthogonal to $B$} to be
\[X(\A_k)^{B }=\bigcap_{\alpha\in B} X(\A_k)^{\alpha}\]
and, for $B\subset \Brnr(X)$, we define the set of \emph{\(\Omega\)-local points orthogonal to \(B\)} to be
\[X(k_\Omega)^{B }=\bigcap_{\alpha\in B} X(k_\Omega)^{\alpha}.\]
When the subgroup $B$ depends functorially on $X$, we often drop $X$ from the notation, writing $X(k_\Omega)^{\Brnr}$ as shorthand for $X(k_\Omega)^{\Brnr(X)}$, for example.

For all $\alpha\in \Brnr(X)$, we have
\begin{equation}\label{eq:rat pts in BM set}
   \ol{X(k)}\subset X(k_{\Omega})^{\alpha}\subset X(k_{\Omega}),  
\end{equation}
where $\ol{X(k)}$ denotes the closure of the rational points in $X(k_{\Omega})$, with respect to the product of the $v$-adic topologies on the $X(k_v)$. 

Now let \(x_v \in (X_{k_v})_0\) be a closed point with residue field $k_v(x_v)$.  
For \(\alpha \in \Brnr(X)\), evaluation at $x_v$ yields
$\alpha(x_v) \in \Br(k_v(x_v))$, and then corestriction $\mathrm{cores}_{k_v(x_v)/k_v} \colon \Br(k_v(x_v)) \to \Br(k_v)$ yields an element in $\Br(k_v)$.
If \(z_v = \sum n_{x_v} x_v\) is a 0-cycle, then we can apply this process to each \(x_v\), 
allowing us to extend the Brauer--Manin pairing to a pairing 
\begin{align*}
\lr{ \cdot, \cdot}_{BM}: Z_{0, \Omega}(X) \times \Brnr(X) & \longrightarrow \Q/\Z\\
\Biggl(\Bigl(\sum_{\substack{x_v \in (X_{k_v})_0}}  n_{x_v} x_v\Bigr)_v\ ,\ \alpha \Biggr) & \longmapsto \sum_{v\in \Omega} \sum_{\substack{x_v \in (X_{k_v})_0
}}  n_{x_v} \text{inv}_v(\mathrm{cores}_{k_v(x_v)/k_v}(\alpha(x_v))).
\end{align*}

\begin{definition}[{\cite{CT-Chow-93}}]
\label{defn:unrefined-Brauer}  
For $\alpha\in \Br_{\mathrm{nr}}(X)$, define the \emph{group of $\Omega$-local 0-cycles on $X$ orthogonal to $\alpha$} to be
$$Z_{0,\Omega}(X)^\alpha \coloneq\left\{ z \in Z_{0,\Omega}(X)\;:\; \lr{z,\alpha}_{BM} = 0 \right\}.$$
For any $B \subset \Br_{\mathrm{nr}}(X)$, we define the set of $\Omega$-local 0-cycles orthogonal to $B$ to be $$ Z_{0,\Omega}(X)^{B} \coloneq\bigcap_{\alpha \in B} Z_{0,\Omega}(X)^\alpha.$$
As above, when the subgroup $B$ depends functorially on $X$, we often drop $X$ from the notation, writing $Z_{0,\Omega}(X)^\Brnr$ as shorthand for $Z_{0,\Omega}(X)^{\Brnr(X)}$, for example. For all $\alpha\in \Brnr(X)$, we have
\[Z_{0}(X)\subset Z_{0,\Omega}(X)^{\alpha}\subset Z_{0,\Omega}(X). \]
\end{definition}

\subsection{Recombining maps and descent sets}\label{sec:descent}
In this section we recall the definitions of recombining maps and (unrefined) obstruction sets for 0-cycles from \cite{balestrieri-berg}.

Let $Y$ be a quasi-projective $k$-variety endowed with an action of an algebraic group $G$ over $k$. Then we denote by $Y^\sigma$ the twist of $Y$ by $\sigma \in \H^1(k,G)$, see \cite[p.12]{Sko-Torsors}.

\begin{definition}\label{defn:f-descent_pt}
Let $G$ be a linear algebraic group over $k$ and let $f: Y\rightarrow X$ be an $X$-torsor under $G$ defined over $k$~\cite[Definition 2.1.1]{Sko-Torsors}. 
For any $\sigma\in \H^1(k,G)$, we have an $X$-torsor under $G^\sigma$ of the form $f^\sigma : Y^\sigma \rightarrow X$.
We can partition the set of $k$-rational points on $X$ as 
\begin{equation}\label{eq:partition}
X(k) = \bigsqcup_{\sigma\in \H^1(k,G)} f^\sigma(Y^\sigma(k)),
\end{equation}
see e.g.\ \cite[\S 2]{Sko-Torsors}.
The \emph{$f$-descent set for rational points} is defined as
\begin{equation*}
 X(\A_k)^f \coloneq \bigcup_{\sigma\in \H^1(k,G)} f^\sigma(Y^\sigma(\A_k)).
\end{equation*}
\end{definition}

\begin{remark}
The inclusions $Y^\sigma(k) \subset Y^\sigma (\A_k)$, together with~\eqref{eq:partition}, yield an inclusion
\begin{equation*}
 X(k) \subset X(\A_k)^f.
 \end{equation*}
\end{remark}

To define analogue of partition~\eqref{eq:partition} for \(0\)-cycles, we first introduce the global and local recombining maps.

\begin{definition}
    Let $G$ be a linear algebraic group over $k$. 
Let $f:Y\rightarrow X$ be an $X$-torsor under $G$ defined over $k$, let $L$ be a finite extension of $k$. 
For any $\sigma\in \H^1(L,G)$, we have an $X_L$-torsor under $G_L^\sigma$ defined over $L$ given by $f_L^\sigma:Y_L^\sigma\rightarrow X_L$. This twisted $X_L$-torsor induces pushforward maps 
\begin{align*}
f_{L,*}^\sigma : Z_0(Y_L^\sigma) \rightarrow Z_0(X_L) \ \textrm{  and  } \ 
f_{L,*}^\sigma : Z_{0,\Omega}(Y_L^\sigma) \rightarrow Z_{0,\Omega}(X_L)
\end{align*}
which preserve degrees of 0-cycles. 
\end{definition}

\begin{definition}
Let $X$ be a variety defined over $k$ and let $L$ be a finite extension of $k$. There exists a natural \emph{corestriction map}
$$N_{L/k}: Z_0(X_L) \rightarrow Z_0(X),$$
which sends a 0-cycle of degree $\delta$ on $X_L$ to a 0-cycle of degree $[L:k]\cdot\delta$ on $X$.

The local corestriction maps $N_{L_w/k_v}: Z_0(X_{L_w}) \rightarrow Z_0(X_{k_v})$ for $w\in\Omega_L$ lying above $v\in\Omega_k$ yield a map \begin{align}\label{eq:adelic cores}
N_{L/k}^\Omega \ :\ & Z_{0,\Omega}(X_L) \rightarrow Z_{0,\Omega}(X)\\
& \nonumber (z_w)_{w\in\Omega_L} \mapsto \bigl(\sum_{w \mid v} N_{L_w/k_v}(z_w)\bigr)_{v\in\Omega_k}
\end{align}
that restricts to $N_{L/k}$ when one views $Z_0(X_L)$ and $Z_0(X)$ as subgroups of $Z_{0,\Omega}(X_L)$ and $Z_{0,\Omega}(X)$, respectively. Explicitly, for $z\in Z_0(X_L)$, 
\begin{equation}\label{eq:locglobnorm}
N_{L/k}^\Omega((z)_{w\in\Omega_L})=(N_{L/k}(z))_{v\in\Omega_k}.
\end{equation}
\end{definition}

\begin{definition}
Let $G$ be a linear algebraic group over $k$ and let $f: Y\rightarrow X$ be an $X$-torsor under $G$ defined over $k$. Define the \emph{global $f$-recombining map} $$ \rec_f : \bigoplus_{L/k \text{ finite}} \bigoplus_{\sigma\in\H^1(L,G)} Z_0(Y_L^\sigma) \rightarrow Z_0(X)$$
in the following way. Let $L_1,\dots, L_n$ be a finite collection of finite extensions of $k$. For each $L_i$ let $\sigma_{i,1},\dots ,\sigma_{i, m_i}$ be a finite collection of elements of $\H^1(L_i,G)$. For a 0-cycle $z\in Z_0(Y_{L_i}^{\sigma_{i,j}})$, the pushforward $f_{L_i,*}^{\sigma_{i,j}}(z)$ lies in $Z_0(X_{L_i})$. We
define $\rec_f(z)\coloneq N_{L_i/k}(f_{L_i ,*}^{\sigma_{i,j}}(z))\in Z_0(X)$ and extend by linearity.
\end{definition}

\begin{remark}
Up to a change of notation, this is the map $g_*$ described in Definition 2.2 of \cite{balestrieri-berg}.
\end{remark}

Using this recombining map, we define the analogue of partition~\eqref{eq:partition} for \(0\)-cycles as follows.

\begin{lemma}\label{lem:partition0cycles}
Let $G$ be a linear algebraic group over $k$ and let $f: Y\rightarrow X$ be an $X$-torsor under $G$ defined over $k$. Then 
\begin{equation*}
Z_0(X) = \rec_f \left( \bigoplus_{\substack{L/k\ 
\mathrm{finite}}}  \bigoplus_{\sigma \in \H^1(L,G)} Z_0(Y_L^\sigma)\right).
\end{equation*}
\end{lemma}
\begin{proof}
See~\cite[Proposition~2.5]{balestrieri-berg}.
\end{proof}

To define descent sets for 0-cycles, we also require the notion of an $\Omega$-local recombining map. 

\begin{definition}
Let $G$ be a linear algebraic group over $k$ and let $f: Y\rightarrow X$ be an $X$-torsor under $G$ defined over $k$. 
Define the \emph{$\Omega$-local $f$-recombining map}
$$ \rec_f^\Omega : \bigoplus_{L/k \text{ finite}} \bigoplus_{\sigma\in\H^1(L,G)} Z_{0,\Omega}(Y_L^\sigma) \rightarrow Z_{0,\Omega}(X)$$
in the following way. Let $L_1,\dots, L_n$ be a finite collection of finite extensions of $k$. For each $L_i$ let $\sigma_{i,1},\dots ,\sigma_{i, m_i}$ be a finite collection of elements of $\H^1(L_i,G)$. For an $\Omega$-local 0-cycle $(z_w)_{w\in\Omega_{L_i}}\in Z_{0,\Omega}(Y_{L_i}^{\sigma_{i,j}})$, define $\rec_f^\Omega((z_w)_{w\in\Omega_{L_i}})\coloneq N_{L_i/k}^\Omega(f_{L_i ,*}^{\sigma_{i,j}}((z_w)_w))\in Z_{0,\Omega}(X)$ and extend by linearity. 
\end{definition}

\begin{remark}
This map is, again up to a change of notation, as described in Definition 3.1 of \cite{balestrieri-berg}.
\end{remark}

\begin{definition}[{\cite[Definition 3.3]{balestrieri-berg}}]
\label{defn:non-refined-f-descent}
Let $G$ be a linear algebraic group over $k$ and let $f:Y\rightarrow X$ be an $X$-torsor under $G$ defined over $k$. The \emph{$f$-descent set for 0-cycles} is defined as 
\begin{equation*}
Z_{0,\Omega}(X)^f \coloneq\rec^\Omega_f \left( \bigoplus_{L/k \textrm{ finite}}  \bigoplus_{\substack{\sigma\in \H^1(L,G)}} Z_{0,\Omega}(Y_L^\sigma)\right).
\end{equation*}
\end{definition}

Note that for $z \in Z_0(Y_L^\sigma)$, we have $\rec_f^\Omega((z)_w) = (\rec_f(z))_v$. Thus, Lemma~\ref{lem:partition0cycles} shows that $Z_0(X)\subset Z_{0,\Omega}(X)^f$.

\section{Refined obstruction sets for 0-cycles}
\label{sec:refined}

Let $k$ be a number field and let $X$ be a smooth, quasi-projective, geometrically integral variety defined over $k$. In this section, we define our new refined obstruction sets for 0-cycles. Consider the free abelian group generated by a set $S$, defined as follows: \[\Z[S] \coloneq\left\{\sum_{s \in S} n_s \;\bigg{\vert}\; n_s \in \Z, n_s = 0\text{ for all but finitely many $s \in S$}\right\}.\]

\begin{definition}
We define the \emph{refined global recombining map for $X$} to be the group homomorphism $$\refrec_X \coloneq \sum_{L/k \text{ finite}} N_{L/k}: \bigoplus_{L/k \text{ finite}} \Z[X(L)] \rightarrow Z_0(X),$$ wherein we view $\Z[X(L)]$ as a subgroup of $Z_0(X_L)$ by viewing points in $X(L)$ as 0-cycles of degree $1$ on $X_L$.
\end{definition}

\begin{remark}\label{rmk:rec surj}
Observe that $Z_0(X) = \im(\refrec_X)$; this justifies the idea of using this refined recombining map to partition the set of 0-cycles on $X$.  
\end{remark}

Now let $L/k$ be a finite extension. The natural map $X(L_\Omega)\to Z_{0,\Omega}(X_L)$ induced by viewing a local point $x_w\in X(L_w)$ as a 0-cycle of degree $1$ in $Z_0(X_{L_w})$ extends by $\Z$-linearity to an injective homomorphism $\Z[X(L_\Omega)] \hookrightarrow Z_{0, \Omega}(X_L)$ making the following diagram commute:
\begin{center}
\begin{equation}\label{eq:rec}
\begin{tikzcd}
\Z[X(L_\Omega)]\ar[rr, hook] && Z_{0, \Omega}(X_L)\\
\Z[X(L)]\ar[u, hook]\ar[rr, hook, ] && Z_{0}(X_L)\ar[u, hook].
\end{tikzcd}
\end{equation}
\end{center}

\begin{remark}\label{rmk:notsurj}
Note that for $z$ in the image of $\Z[X(L)] \hookrightarrow Z_0(X_L)$, each closed point in the support of $z$ has degree 1, and therefore this map is not surjective. Similarly, the map $\Z[X(L_\Omega)] \hookrightarrow Z_{0, \Omega}(X_L)$ is not surjective.
\end{remark}

\begin{definition}\label{defn:refinedlocalrecombiningmap}
We define the \emph{refined $\Omega$-local recombining map} for $X$ to be the group homomorphism
\begin{align*}
    \refrec_{X, \Omega} \coloneq \sum_{L/k \ \textrm{finite}} N_{L/k}^\Omega 
    \colon \bigoplus_{L/k \ \textrm{finite}} \Z[X(L_\Omega)] &\rightarrow Z_{0, \Omega}(X),
\end{align*}
wherein we view $\Z[X(L_\Omega)]$ as a subgroup of $Z_{0, \Omega}(X_L)$ by viewing elements of $X(L_\Omega) $ as 
$\Omega$-local 0-cycles of degree $1$ on $X_L$.
\end{definition}

By~\eqref{eq:locglobnorm} and~\eqref{eq:rec}, the inclusions $X(L)\hookrightarrow X(L_\Omega)$ and $Z_0(X)\hookrightarrow Z_{0, \Omega}(X)$ give a commutative diagram
\begin{center}
\begin{equation}\label{eq:recX}
\begin{tikzcd}
\bigoplus_{L/k \ \textrm{finite}} \Z[X(L_\Omega)]\ar[rr, "\refrec_{X,\Omega}"] && Z_{0, \Omega}(X)\\
\bigoplus_{L/k \ \textrm{finite}} \Z[X(L)]\ar[u, hook]\ar[rr, ->>, "\refrec_{X}"] && Z_{0}(X)\ar[u, hook]. 
\end{tikzcd}
\end{equation}
\end{center}

\begin{remark}
Any 0-cycle $z$ can be written uniquely as $z = z^+ - 
z^-$, where \(z^{+}\) and \(z^{-}\) are effective. One benefit of the refined recombining map is that it enables one to use this splitting and exploit the correspondence between effective 0-cycles of degree $d$ and rational points on $\text{Sym}^d(X)$ (see e.g.\ \cite{VirayVogt2024}). 
\end{remark}

\begin{definition}
\label{defn:refined-obstruction} 
Consider a collection $\{ \mathscr{X}^{(L)} \}_{L/k}$ of 
obstruction sets for $X$, where $L/k$ ranges over the finite extensions of $k$; that is, for each finite extension $L/k$, we have $X(L) \subset \mathscr{X}^{(L)} \subset X(L_{\Omega})$. 
We define the \emph{refined obstruction set for 0-cycles with respect to $\{ \mathscr{X}^{(L)} \}_{L/k}$} to be
\[ 
\RefinedObs{X}^{\{ \mathscr{X}^{(L)} \}_{L/k}} \coloneq\refrec_{X, \Omega}\left(\bigoplus_{L/k \ \textrm{finite}}\Z[\mathscr{X}^{(L)}]\right) \subset Z_{0,\Omega}(X).
\]
For $\delta\in\Z$, we define $\RefinedObsDeg{X}{\delta}^{\{ \mathscr{X}^{(L)} \}_{L/k}}\coloneq\RefinedObs{X}^{\{ \mathscr{X}^{(L)} \}_{L/k}}\cap Z_{0,\Omega}^\delta(X).$
\end{definition}

\begin{remark}
\
\begin{enumerate}[(i)]
    \item For the applications in this paper, we do not require the full flexibility of Definition~\ref{defn:refined-obstruction}: our obstruction sets $\mathscr{X}^{(L)}$ are defined in a uniform way, using some obstruction for rational points (for example the Brauer--Manin obstruction). In these cases we use more compact notation for the set $\RefinedObs{X}^{\{ \mathscr{X}^{(L)} \}_{L/k}}$, replacing the collection $\{ \mathscr{X}^{(L)} \}_{L/k}$ by
    the defining obstruction, cf.\ Examples~\ref{example:obstruction-sets-main}.
    \item On the other hand, Definition~\ref{defn:refined-obstruction} can be generalised even further, by expanding the notion of an obstruction set, and this we do exploit: see Definition~\ref{defn:generalised-refined-obstruction} and the results in Section~\ref{sec:corwin-schlank}. 
\end{enumerate}
\end{remark}

\begin{examples}
\label{example:obstruction-sets-main}
\begin{enumerate}[(i)]
    \item Let $G$ be a linear algebraic group over $k$ and let $f \colon Y \rightarrow X$ be an $X$-torsor under $G$ defined over $k$. The \emph{refined $f$-descent set for 0-cycles} is \[\RefinedObs{X}^{f} = \refrec_{X, \Omega}\left(\bigoplus_{L/k \ \textrm{finite}}\Z[X(\A_L)^{f}]\right)  .\]

    \item\label{ex:part:BM} The \emph{refined Brauer--Manin set for $0$-cycles} is \[\RefinedObs{X}^{\Br} = \refrec_{X, \Omega}\left(\bigoplus_{L/k \ \textrm{finite}}\Z[X(\A_L)^{\Br}]\right)  .\]

    \item\label{ex:part:Unram-BM} The \emph{refined unramified Brauer--Manin set for $0$-cycles} is \[\RefinedObs{X}^{\Brnr} = \refrec_{X, \Omega}\left(\bigoplus_{L/k \ \textrm{finite}}\Z[X(L_\Omega)^{\Brnr}]\right)  .\]

    \item\label{ex:part:descent} The \emph{refined descent set for $0$-cycles} is \[\RefinedObs{X}^{\operatorname{desc}} = \refrec_{X, \Omega}\left(\bigoplus_{L/k \ \textrm{finite}}\Z[X(\A_L)^{\operatorname{desc}}]\right)  ,\] where 
    $$X(\A_L)^{\text{desc}} \coloneq \bigcap_{ \substack{G \textrm{ lin alg group}\\\textrm{over $L$}}} \bigcap_{[f: Y \to X_L] \in \H^1(X_L, G)} X(\A_L)^{f}.$$
    We note that $\RefinedObs{X}^{\operatorname{desc}}$  is \emph{not} defined as $ \bigcap_{ \substack{G \textrm{ lin alg group}\\\textrm{over $k$}}} \bigcap_{[f: Y \to X] \in \H^1(X, G)} \RefinedObs{X}^{f}$. 
    We can also define the \emph{refined connected descent set for 0-cycles} $\RefinedObs{X}^{\conn}$ in a similar way to $\RefinedObs{X}^{\operatorname{desc}}$ by only considering connected linear algebraic groups, the \emph{refined finite descent set for 0-cycles} $\RefinedObs{X}^{\et}$ by only considering finite linear algebraic groups, and so on. 
    
    \item\label{ex:part:etale-Br} The \emph{refined étale-Brauer set for $0$-cycles} is \[\RefinedObs{X}^{\et,\Br} = \refrec_{X, \Omega}\left(\bigoplus_{L/k \ \textrm{finite}}\Z[X(\A_L)^{\et, \Br}]\right),\]
    where  
    $$X(\A_L)^{\et, \Br} \coloneq \bigcap_{ \substack{G \textrm{ finite lin alg group}\\\textrm{over $L$}}} \bigcap_{[f: Y \to X_L] \in \H^1(X_L, G)} \bigcup_{\sigma \in \H^1(L, G)} f^{\sigma}(Y^{\sigma}(\A_L)^{\Br}).$$
\end{enumerate}
\end{examples}

\begin{lemma}\label{lem:globalinobs} 
Let $\{ \mathscr{X}^{(L)} \}_{L/k}$ be a collection of  obstruction sets for $X$. Then the natural inclusion $Z_0(X) \subset Z_{0,\Omega}(X)$ induces an inclusion $Z_0(X) \subset \RefinedObs{X}^{\{ \mathscr{X}^{(L)} \}_{L/k}}$. 
\end{lemma}

\begin{proof}
This follows from Remark~\ref{rmk:rec surj} and the commutative diagram~\eqref{eq:recX}. 
\end{proof}

\begin{remark}
\label{remark:refined-vs-non-refined}
The reason for the terminology `refined' is that there exist other obstructions to 0-cycles in the literature which are \emph{a priori} coarser. 
For instance, for a quasi-projective variety \(X\),
\[
\RefinedObs{X}^{\Brnr} \subset Z_{0,\Omega}(X)^{\Brnr},
\]
where \(Z_{0,\Omega}(X)^{\Brnr}\) is defined as in Definition~\ref{defn:unrefined-Brauer} with $B=\Brnr(X)$. 
To give another example, if $f: Y \to X$ is a torsor under some linear algebraic group $G$ over $k$, then we have a natural inclusion 
\[\RefinedObs{X}^{f} \subset Z_{0,\Omega}(X)^{f}.\] Moreover, statements involving multiple torsors are true in some cases, see e.g.\ Lemma \ref{lem:1} and Remark~\ref{rem:1}.
\end{remark}

\section{Hasse principle and weak approximation for 0-cycles} 
\label{subsec:prelim-weak-approximation}
Let $k$ be a number field and let $X$ be a smooth, quasi-projective, geometrically integral variety defined over $k$. It is easy to generalise the Hasse principle for rational points to the setting of 0-cycles: 
we say that the Hasse principle for 0-cycles of degree $\delta$ holds if $Z^{\delta}_{0,\Omega}(X)\neq \emptyset$ implies $Z^\delta_0(X)\neq \emptyset$. 

Now recall that weak approximation for rational points holds if \(X(k)\) is dense in \(\prod_{v \in \Omega} X(k_v)\) with respect to the product topology. Explicitly, for any \((x_v)_v \in \prod_{v \in \Omega} X(k_v)\), and any open set \(U \subset \prod_{v \in \Omega} X(k_v)\) containing \((x_v)_v\), there is a rational point \(y \in X(k)\) such that the image of \(y\) under the map \(X(k) \hookrightarrow \prod_{v \in \Omega} X(k_v)\) lies in \(U\). 
In the product topology, open sets are of the form \(U = \prod_v U_v\) where $U_v$ is open in $X(k_v)$ for all $v\in\Omega$ and \(U_v = X(k_v)\) for all but finitely many \(v \in \Omega\).

In order to define weak approximation for 0-cycles, we first recall the definition of the Suslin homology group of degree $0$. We will use the notation from \cite[Section 3]{balestrieri-berg} for most of this section. We refer the reader to \cite[Section 5]{Sch07} and \cite{ES08} for further exposition of Suslin homology.
Recall that the group of finite correspondences \(\Cor(V,W)\) between two \(k\)-varieties \(V\) and \(W\) is the set of $\Z$-linear combinations of integral subschemes of \(V \times_k W\) that are finite and flat over \(V\). For the $k$-variety \(X\), \(\Cor(\spec k, X)\) is generated by integral subschemes \(Z \subset X\) finite and flat over \(\spec k\), i.e.\ by closed points on \(X\). Thus, \(\Cor(\spec k, X)\) is the group of 0-cycles on \(X\).
For an integral subscheme \(W \subset \A^1 \times_k X\) in \(\Cor(\A^1, X)\),
any point \(\lambda : \spec k \to \A^1\) defines an integral subscheme \(W_{\lambda} \in \Cor(\spec k, X)\) via pullback along \(X \xrightarrow{(\lambda, id)} \A^1 \times_k X\). This map can be extended linearly. Fixing \(0, 1 \in \A^1\) we get a map 
\[
\Phi \colon \Cor(\A^1, X) \rightarrow \Cor(\spec k, X); \qquad \sum n_W W \mapsto \sum n_W\left( W_0 - W_1\right).
\]

\begin{definition}\label{def suslin hom}
The \emph{Suslin homology group of degree 0 on \(X\)} is defined as the quotient
\[
h_0(X) = Z_0(X)/\im(\Phi).
\]
\end{definition}

Alternatively (see~\cite[Theorem~5.1]{Sch07}), $h_0(X)$ is the quotient of $Z_0(X)$ by the subgroup generated by elements of the form $\divisor( g )$, where $g\in k(C)$ is a rational function on a closed integral curve $C$ on $X$ such that $g$ is defined and $g\equiv 1$ at every point of $\widetilde{C}^c\setminus \widetilde{C}$. Here, $\widetilde{C}$ denotes the normalisation of $C$ in $k(C)$, and $\widetilde{C}^c$ denotes its smooth compactification. Note that when \(X\) is proper over \(k\), the Suslin homology group $h_0(X)$ is the same as the Chow group $\CH_0(X) := Z_0(X)/\sim$, where $\sim$ denotes rational equivalence of 0-cycles on \(X\).

\begin{definition}[Weak approximation topology]\label{def:weakapproxtop}
Let \((0)_v \in Z_{0, \Omega}(X)\) denote the trivial $\Omega$-local 0-cycle. For $n\in\Z_{>0}$ and a finite subset $S\subset \Omega$, let     
\[
\mathcal{A}_{n,S}
\coloneq\left\{(z_v)_v\in Z_{0,\Omega}(X):
 z_v \text{ maps to }0 \text{ in } h_0(X_{k_v})/n \text{ for all }v\in S\right\}.
\]
Note that if \(n \mid m\) and \(T \subset S\), then 
\begin{align*}
\label{eq:open-inclusions}
\mathcal{A}_{m,S} \subset \mathcal{A}_{n,T}.
\end{align*}
The sets $\mathcal{A}_{n,S}$ form a neighbourhood basis of \((0)_v\) defining the \emph{weak approximation topology} on $Z_{0,\Omega}(X)$. 
\end{definition}

From now on, for a number field \(k\) with set of places \(\Omega\), we will equip \(Z_{0, \Omega}(X)\) with the weak approximation topology as defined in Definition~\ref{def:weakapproxtop}. 

\begin{definition}[Weak approximation for 0-cycles]\label{Def WA for 0cycles}
We say that \emph{\(X\) satisfies weak approximation for 0-cycles of degree \(\delta\)} if the following holds: given \(n \in \Z_{>0}\), a finite set \(S \subset \Omega\), and \((z_v)_v \in Z_{0, \Omega}^\delta(X)\), there is a global 0-cycle \(\globcyc \in Z_0^\delta(X)\) such that \(\globcyc\) and \(z_v\) have the same image in \(h_0(X_{k_v})/n\) for all \(v \in S\). In other words, $\overline{Z_0^\delta(X)} = Z_{0, \Omega}^\delta(X)$, where the closure is taken
with respect to the weak approximation topology. Similarly, we say that \emph{$X$ satisfies weak approximation for $0$-cycles} if \emph{$\overline{Z_0(X)} = Z_{0, \Omega}(X)$}.
\end{definition}

Proposition 4.4 of \cite{balestrieri-berg} gives a more concrete way of getting to grips with weak approximation. 

\begin{definition}[{Sufficiently close points: \cite[Proposition~1.1(3)]{liang-principelocalglobal-11}}]
Let \(v \in \Omega \), let \(Y\) be a variety over \(k_v\). A closed point \(y \in Y\) can be viewed as a point in $Y(k_v(y))$, where $k_v(y)$ denotes the residue field of $y$. Let \(U_y \subset Y(k_v(y))\) denote an open neighbourhood of \(y\) with respect to the \(v\)-adic analytic topology.
Another closed point \(y' \in Y\) is called \emph{sufficiently close to \(y\) with respect to \(U_y\)} if there exists an isomorphism $k_v(y')\cong k_v(y)$ such that, viewing \(y'\) as a point in \(Y(k_v(y))\) via this isomorphism, we have \(y' \in U_y \subset Y(k_v(y))\).
This definition is extended linearly to 0-cycles on \(Y\), but now with respect to a system of open neighbourhoods of closed points in the support of the 0-cycle. 
\end{definition}

\begin{proposition}[{\cite[Proposition 4.4]{balestrieri-berg}}]
\label{prop:BB-Wittenberg-rephrasing-WA}
Let \(z_v \in Z^{\delta}_{0}(X_{k_v})\) and let \(n \in \Z_{>0}\). Then there is a system of open neighbourhoods of the points in the support of $z_v$ such that, if \(z_v' \in Z^{\delta}_{0}(X_{k_v})\) is sufficiently close to \(z_v\) with respect to this system, then \(z_v\) and \(z_v'\) have the same image in \(h_0(X_{k_v})/n\).
\end{proposition}

Note that this can be extended to a finite set \(S\) of places, and also motivates the definition of the weak approximation topology.  The proof of this proposition relies upon the fact (proved by Wittenberg, see~\cite[Proposition~4.6]{balestrieri-berg}) that for any \(n\) and any \(d\), the map 
\[
\sym^d(X_{k_v}) \to h_0(X_{k_v})/n
\]
is locally constant.

We now give an immediate general result related to weak approximation for 0-cycles using refined obstruction sets, given knowledge of the behaviour of rational points over finite extensions of the base field.

\begin{proposition} \label{prop:propofrefobs}
Let $X$ be a smooth, quasi-projective, geometrically integral variety over a number field $k$. For each finite extension $L/k$, let $\mathscr{X}^{(L)}\subset X(L_{\Omega})$ be an obstruction set equipped with either the adelic or product topology, such that $X(L)$ is dense in $\mathscr{X}^{(L)}$ with respect to this topology. 
Then $\RefinedObs{X}^{\{ \mathscr{X}^{(L)} \}_{L/k}} \subset \overline{Z_0(X)}$, where the closure is with respect to the weak approximation topology for 0-cycles. 
Moreover, if we fix any degree $\delta$, we have $\widetilde{Z}_{0, \Omega}^{\delta}(X)^{\{ \mathscr{X}^{(L)} \}_{L/k}} \subset \overline{Z^{\delta}_0(X)}$.   
\end{proposition}

\begin{proof}
Let $S \subset \Omega_k$ be a finite non-empty set of places and let $n \in \Z_{>0}$. Let $(z_v)_v \in \RefinedObs{X}^{\{ \mathscr{X}^{(L)} \}_{L/k}}$. By Definition~\ref{defn:refined-obstruction}, there is a finite non-empty set $\mathscr{F}$ of finite extensions of $k$ and, for each $L \in \mathscr{F}$, a finite set $\mathscr{P}_L$ of $\Omega_L$-local points of the form $(x_w)_{w \in \Omega_L} \in \mathscr{X}^{(L)}$ such that, for any $v \in \Omega_k$,  $z_v$ can be written as 
\[ 
z_v = \sum_{L \in \mathscr{F}} \sum_{(x_w)_{w \in \Omega_L} \in \mathscr{P}_L} n_{(x_w)_w} \sum_{w|v} N_{L_{w}/ k_v}(x_w)
\]
for some integers $n_{(x_w)_w}  \in \Z$.
By assumption, for each $L \in \mathscr{F}$ and each $(x_w)_{w \in \Omega_L} \in \mathscr{P}_L$, there is a global point $x_{L, (x_w)_w} \in X(L)$ such that $x_w$ and $x_{L, (x_w)_w}$ are sufficiently close (implying, in particular, that they have the same image in $h_0(X_{L_w})/n$ for all $w \in \Omega_L$ above the places in $S$). 
Hence, if we consider the global 0-cycle
\[ \mathbf{z}\coloneq   \sum_{L \in \mathscr{F}} \sum_{(x_w)_{w \in \Omega_L} \in \mathscr{P}_L} n_{(x_w)_w}  N_{L/ k}(x_{L, (x_w)_w})\]
it follows that $\mathbf{z}$ 
and $z_v$ have the same image in $h_0(X_{k_v})/n$ for all $v \in S$, as required.  
If $(z_v)_v \in \widetilde{Z}_{0, \Omega}^{\delta}(X)^{\{ \mathscr{X}^{(L)} \}_{L/k}}$, then $\deg(z_v) = \delta$ for all $v \in \Omega$, and it is straightforward to check that $\deg(z) = \delta$ as well.
\end{proof}

\section{Locally constant Brauer classes}
\label{sec:locconst}

In this section, we recall various alternative definitions of the subgroup of locally constant Brauer classes and collect some results on obstruction sets for torsors under abelian varieties.

Let $k$ be a number field and let $X$ be a smooth, quasi-projective, geometrically integral variety defined over $k$.
Recall that \(\Br_1(X) \coloneq\ker(\Br(X) \to \Br(X_{\ol{k}}))\) and \(\Br_0(X) \coloneq\im(\Br(k) \to \Br(X))\subset \Br_1(X)
\). 
The Hochschild--Serre spectral sequence gives rise to the homorphisms \(r\) and \(r_v\) in the following commutative diagram

\begin{center}
\begin{tikzcd}
     0\arrow[d]& & 
     \\
    \Sha(k, \pic^0(X_{\ol{k}}) ) \arrow[d]& & & &
    \\
    \H^1(k, \pic^0(X_{\ol{k}})) \arrow[r, "i_*"] \arrow[d]& \H^1(k, \pic(X_{\ol{k}})) \arrow[d] & {\Br_1(X)} 
    \arrow[l, "r"'] \arrow[d] & {\Br_0(X)}\ar[d] \arrow[l] & {0} \arrow[l] \\
    \prod_v \H^1(k_v, \pic^0(X_{\ol{k}_{v}})) \arrow[r, "i_*"]&  \prod_v\H^1(k_v, \pic(X_{\ol{k}_{v}})) & {\prod_v \Br_1(X_{k_v})}
    \arrow[l, "\prod r_v"'] & {\prod_v \Br_0(X_{k_v})\arrow[l]} & {\arrow[l] 0,}
\end{tikzcd}
\end{center}
where the maps $i_*$ are induced by the inclusions  $i:\pic^0(X_{\ol{k}})\subset \pic(X_{\ol{k}})$ and $i:\pic^0(X_{\ol{k}_v})\subset \pic(X_{\ol{k}_{v}})$. 

Let \(\Br_{1/2}(X)=  r^{-1}\left(i_*(\H^1(k, \pic^0 (X_{\ol{k}}) ) \right) \). There are several definitions of `the subgroup of locally constant Brauer classes' in the literature, each of them contravariantly functorial in the variety $X$. For instance, Creutz (\cite{Creutz-torsors}) defines this subgroup to be 
\[
\locconstBr_{Cr}(X) \coloneq \ker \left( \Br_{1/2}(X) \to \prod_{v} \Br_{1}(X_{k_v})/\Br_0(X_{k_v})\right).
\]
On the other hand, Skorobogatov (\cite[Part two, Notation]{Sko-Torsors}) uses \[\locconstBr_{Sk}(X)\coloneq\ker \left( \Br_1(X) \to \prod_{v} \Br_1(X_{k_v})/\Br_0(X_{k_v})
   \right),\] while Manin (\cite[Th\'{e}or\`{e}me~6]{Manin71}) uses \[\locconstBr_{Ma}(X) \coloneq r^{-1}(i_*(\Sha(k, \pic^0 (X_{\ol{k}}))).\] 
With this notation, we have 
\begin{equation}\label{eq:subgroups of Brnr}
  \locconstBr_{Ma}(X) \subset \locconstBr_{Cr}(X) \subset \locconstBr_{Sk}(X)\subset \Brnr(X),  
\end{equation}
where the rightmost inclusion follows from~\cite[(6.1.4)]{Sansuc81}. In particular, if $X$ is proper, then 
\begin{equation*}
\begin{tikzcd}
\locconstBr_{Ma}(X) \ar[r, phantom, "\subset"]&\locconstBr_{Cr}(X) \ar[r, phantom, "\subset"] \ar[dr, phantom, sloped, "\subset" ] &\locconstBr_{Sk}(X)\ar[r, phantom, "\subset"]& \Br_1(X)\ar[r, phantom, "\subset"]&\Br(X)\\
& & \Br_{1/2}(X) \ar[ur, phantom, sloped, "\subset" ] & &
\end{tikzcd}
\end{equation*}
and hence
\begin{equation}\label{eq: russian B sets}
\begin{tikzcd}
X(k_{\Omega})^{\Br(X)} \ar[r, phantom, "\subset"]&
X(k_{\Omega})^{\Br_1(X)} \ar[r, phantom, "\subset"] \ar[dr, phantom, sloped, "\subset"]& X(k_{\Omega})^{\locconstBr_{Sk}(X)} \ar[r, phantom, "\subset"]&X(k_{\Omega})^{\locconstBr_{Cr}(X)} \ar[r, phantom, "\subset"]&X(k_{\Omega})^{\locconstBr_{Ma}(X)}\\
& & X(k_{\Omega})^{\Br_{1/2}(X)} \ar[ur, phantom, sloped, "\subset"] & &
\end{tikzcd}
\end{equation}

\begin{remark}\label{remark:locconstBr finite}
Suppose that $X$ is proper and let $A=\pic^0(X)$. By \cite[Proposition 2.14]{elementary}, if $\Sha(A_{L})$ is finite for every finite extension $L/k$ then $\locconstBr_{Sk}(X_L)/\Br_0(X_L)$ is finite, and hence so are $\locconstBr_{Cr}(X_L)/\Br_0(X_L)$ and $\locconstBr_{Ma}(X_L)/\Br_0(X_L)$. The finiteness of these subquotients of $\Br(X_L)$ can be advantageous when computing the associated obstruction sets $X(L_{\Omega})^{\locconstBr_{Sk}(X_L)}, X(L_{\Omega})^{\locconstBr_{Cr}(X_L)}$ and $ X(L_{\Omega})^{\locconstBr_{Ma}(X_L)}$, since the Brauer--Manin pairing factors through $\Br(X_L)/\Br_0(X_L)$.
\end{remark}

We now collect some useful results on sufficiency of the Brauer--Manin obstruction for local-global principles on $k$-torsors under abelian varieties with finite Tate--Shafarevich groups.

\begin{theorem}
\label{thm:HP/WA-torsors}
Let \(A\) be an abelian variety over \(k\) such that \(\Sha(A)\) is finite. Let \(Y\) be a \(k\)-torsor under \(A\). 
\begin{enumerate}[(i)]
\item\label{part:HP/WA-torsors-1}  \( \overline{Y(k)} = Y(k_{\Omega})^{\Br_1(Y)}\) $($\cite[Theorem 1]{Harari-torsors-06}$)$, i.e.\ all failures of the Hasse principle or weak approximation for rational points on $Y$ are explained by (algebraic) Brauer--Manin obstructions. 
\item\label{part:HP/WA-torsors-2} If \(Y(k_{\Omega})^{\locconstBr_{Ma}(Y)} \neq \emptyset\), then \(Y(k) \neq \emptyset\) $($\cite{Manin71}$)$, i.e.\ all failures of the Hasse principle for rational points on $Y$ are explained by locally constant Brauer--Manin obstructions. 
\end{enumerate}
\end{theorem}

\section{Generalised Kummer varieties}
\label{section:kummer-descent}

In this section, we recall the definition of generalised Kummer varieties following  \cite{Zhu-GenKummer-25, Sko-Zarhin-Kummer-17}. 
Let $k$ be a number field, let $A$ be an abelian variety of dimension at least two over $k$, and let $\varphi\in\Aut_k(A)$ be a non-trivial automorphism of finite order. 
Then \(\ker[1- \auto] = \{x \in A \mid \auto(x) = x\} \) is a subgroup scheme of \(A\). 
Let $G=\lr{\auto}$ and suppose that for all $g\in G\setminus\{1\}$, the subgroup scheme $\ker[1-g]$ has finitely many points over an algebraic closure of $k$.

Let \(T\) be a $k$-torsor under $\ker[1-\auto]$ representing a class in \(\H^1(k, \ker[1-\auto])\). The image of $T$ under the map \(\H^1(k, \ker[1-\auto]) \to \H^1(k,A)\) is represented by  
\[
Y = (A \times_k T)/\ker[1-\auto],
\]
cf.~\cite[Proposition~3.3.2]{Sko-Torsors}.
The variety \(Y\) is a $k$-torsor under \(A\) and is in particular projective. Moreover, different choices of \(T\) representing the same class in \(\H^1(k, \ker[1 - \auto])\) yield isomorphic choices for \(Y\).
Furthermore, by construction,  \(\auto\) acts trivially on \(\ker[1-\auto]\) and \(T\) and thus descends to an automorphism of \(Y\), which we will also denote by \(\auto\). Consider the geometric quotient \(W \coloneq Y/G\). The quotient map \(f \colon Y \to W\) is \'{e}tale away from \(\Fix_{G}(Y) =\{y\in Y\mid g(y)=y \text{ for some } g\in G\setminus\{1\} \}\), and the singularities of $W$ are contained in the image of \(\Fix_{G}(Y)\). In particular, there are dense open subschemes \(U \subset W\) and \(V \subset Y\) such that the map \(f \colon V \to U\) is an \'{e}tale torsor under $G$.

\begin{definition}\label{defn:Kummer}
A \emph{generalised Kummer variety associated to \((A,\varphi, T)\)} is a desingularisation $X$ of \(W\). 
\end{definition}

The following diagram summarises the construction. 
\begin{center}
\begin{equation}\label{eq:Kummer}    
\begin{tikzcd}
& Y \coloneq (A\times_k T)/\ker[1-\auto] \arrow[dd, "f"]   & V \coloneq Y\setminus \text{Fix}_{G}(Y) \arrow[dd, "f"] \arrow[l, phantom, "\supset", hook'] \\                             \\
X\ar[r] & W\coloneq Y/G                                         &   U \coloneq X\setminus f\left(\Fix_{G}(Y)\right) 
\arrow[l, phantom, "\supset", hook']         
\end{tikzcd}
\end{equation}
\end{center}

\begin{remark}[Choice of desingularisation]
When \(W\) admits a minimal desingularisation (e.g. when \(\auto = [-1]\) or when \(A\) is an abelian surface), this can be taken to define \emph{the} generalised Kummer variety associated to \((A, \auto, T)\). In general, one may have different choices of desingularisation. However, this choice will not make a difference to our results, since our proofs proceed via an auxiliary result for \(U\), which is common to all of these birational models.
\end{remark}

\begin{remark}[Automorphisms of prime order]
When \(\varphi\) has prime order, \(\Fix_{G}(A) = \ker[1 - \auto]\), so \(\Fix_{G}(Y)\) is a $k$-torsor under \(\ker[1 - \auto]\). In this case, a generalised Kummer variety can be constructed by lifting \(\auto\) to the blow up \(\Bl_{\Fix_{G}(Y)} \to Y \), and then taking the quotient by the action of this lift. If \(\auto\) is the involution \([-1] \colon A \to A\), then \(\ker[1- \auto] = A[2]\) and this is classical (see e.g.\ \cite[\S 2]{Sko-Zarhin-Kummer-17}). For general primes, see \cite{Zhu-GenKummer-25}.
\end{remark}

\subsection{Obstruction sets for generalised Kummer varieties}\label{sec:Xobs}

For a smooth, quasi-projective, geometrically integral $k$-variety $Z$, let $\mathcal{B}(Z)$ be one of the subgroups of $\Brnr(Z)$ in~\eqref{eq:subgroups of Brnr}, so $\mathcal{B}\in \{\locconstBr_{Ma},
\locconstBr_{Cr}, \locconstBr_{Sk}, \Brnr\}$.

\begin{definition}\label{defn:f-Brnr_pt}
Let $k$ be a number field, let $X$ be a smooth, quasi-projective, geometrically integral variety over $k$.
Let $G$ be a linear algebraic group over $k$ and let $f: Y\rightarrow X$ be an $X$-torsor under $G$. Define the \emph{$(f,\mathcal{B})$-obstruction set for rational points} to be 
\begin{equation*}
 X(k_\Omega)^{f,\mathcal{B}} \coloneq \bigcup_{\sigma\in \H^1(k,G)} f^\sigma(Y^\sigma(k_\Omega)^{\mathcal{B}}).
\end{equation*}
By~\eqref{eq:rat pts in BM set} and~\eqref{eq:partition}, we have 
\begin{align}
\label{eqn:X-descent-fB}
X(k) = \bigsqcup_{\sigma \in \H^1(k,G)} f^{\sigma}(Y^{\sigma}(k)) \subset  \bigcup_{\sigma \in \H^1(k,G)} f^{\sigma}\left(Y^{\sigma}(k_{\Omega})^{\mathcal{B}}\right) = X(k_{\Omega})^{f, \mathcal{B}}. 
\end{align}
\end{definition}

We now return to the setting of~\eqref{eq:Kummer}. Note that $U$ can be viewed as a dense open subscheme of the generalised Kummer variety $X$, since \(U\) is contained in the smooth locus of \(W\). 
By construction, $f : V \to U$ is an \'etale torsor under $G$ so that, by \eqref{eqn:X-descent-fB},
for any field extension \(L/k\) we have
\[U(L) \subset U(L_{\Omega})^{f, \mathcal{B}} \subset U(L_\Omega).\]
Thus, as in Definition~\ref{defn:refined-obstruction}, we can define the refined obstruction set 
\begin{equation}\label{eq:fBUobs} \widetilde{Z}_{0,\Omega}(U)^{f,\mathcal{B}}\coloneq\refrec_{X, \Omega}\left(\bigoplus_{L/k \ \textrm{finite}}\Z[U(L_{\Omega})^{f,\mathcal{B}}]\right)\subset Z_{0,\Omega}(U),\
\end{equation}
and $\widetilde{Z}^\delta_{0,\Omega}(U)^{f,\mathcal{B}}\coloneq \widetilde{Z}_{0,\Omega}(U)^{f,\mathcal{B}} \cap Z_{0, \Omega}^\delta(U)$. We will use $\widetilde{Z}^\delta_{0,\Omega}(U)^{f,\mathcal{B}}$ to define an obstruction set for the generalised Kummer variety $X$. For this, we will need the following moving lemma.

\begin{lemma}[Moving Lemma {\cite[Special case of Proposition~5.6]{Sch07}}]
\label{lem:movinglemmasuslin}
Let $X$ be a smooth variety over \(k\) and let $\iota: U\hookrightarrow X$ be a dense open subscheme. Then $\iota_*:Z_0(U)\to Z_0(X)$ induces a surjective homomorphism $ h_0(U)\twoheadrightarrow h_0(X)$. 
\end{lemma}

\begin{corollary} \label{cor:samedeg} 
Let $X$ be a smooth variety over \(k\) and let $\iota: U\hookrightarrow X$ be a dense open subscheme. 
Then \[Z_0^{\delta}(U)= \emptyset \iff Z_0^{\delta}(X)=\emptyset.\]
\end{corollary}

\begin{proof}
One direction is immediate, since $Z_0^{\delta}(U) \subset Z_0^{\delta}(X)$. For the other direction, let $z \in Z_0^{\delta}(X)$. By the Moving Lemma~\ref{lem:movinglemmasuslin}, there exists $u\in Z_0(U)$ such that $\iota_*[u]=[z]\in h_0(X)$. Since the degree map $\deg : Z_0(X) \to \Z$ factors through $h_0(X)$ (as any 0-cycle in the kernel of the natural projection $Z_0(X)\twoheadrightarrow h_0(X)$ has degree zero), we have $\deg(u)=\deg(z)=\delta$.
\end{proof}

\begin{lemma}\label{lem:globalinobsX}
In the setting of~\eqref{eq:Kummer},
we have
\[\ol{Z_0^\delta(X)}\subset \ol{Z_0^\delta(U)}\subset\overline{\widetilde{Z}^\delta_{0,\Omega}(U)^{f,\mathcal{B}}},\] where the closure is taken with respect to the weak approximation topology on $Z_{0,\Omega}(X)$ (see Definition \ref{def:weakapproxtop}).
\end{lemma}

\begin{proof}
Lemma~\ref{lem:globalinobs} gives \(Z_0^\delta(U)\subset \widetilde{Z}^\delta_{0,\Omega}(U)^{f,\mathcal{B}}\) and therefore
\[\ol{Z_0^\delta(U)}\subset \ol{\widetilde{Z}^\delta_{0,\Omega}(U)^{f,\mathcal{B}}}.
\]
Thus, it suffices to show that \(Z_0^\delta(X)\subset \ol{Z_0^\delta(U)}.\) 
Let $z\in Z_0^\delta(X)$. By the proof of Corollary \ref{cor:samedeg}, there exists $u\in Z_0(U)$ such that $\iota_*[u]=[z]\in h_0(X)$ and $\deg(u)=\deg(z)=\delta$.
\end{proof}

\begin{remark}
Lemma \ref{lem:globalinobsX} shows that \[Z_0^\delta(X) \subset \overline{\widetilde{Z}^\delta_{0,\Omega}(U)^{f,\mathcal{B}}} \subset Z_{0,\Omega}^\delta(X),\] which means we can consider $\overline{\widetilde{Z}^\delta_{0,\Omega}(U)^{f,\mathcal{B}}}$ as an obstruction set for local-global principles for 0-cycles of degree $\delta$ on $X$.
\end{remark}

\section{Local-global principles for 0-cycles on generalised Kummer varieties}\label{sec:weak-approx}

In this section, we prove Theorems \ref{thm:main intro-weak-approx} and \ref{thm:main intro-hasse-principle}. Throughout the section, we retain the notation and assumptions of Section \ref{section:kummer-descent}: let \(A\) be an abelian variety of dimension at least two over a number field \(k\). Let $\varphi\in\Aut_k(A)$ be a non-trivial automorphism of finite order such that each non-trivial element of \(\lr{\auto}\) acts with finitely many fixed points on \(A(\ol{k})\), and write $G=\lr{\auto}$. 
Let $T$ be a $k$-torsor under $\ker[1 - \varphi]$ and let $X$ be a generalised Kummer variety associated to $(A, \varphi, T)$, with $Y$ and $f:V\to U$ as in~\eqref{eq:Kummer}.
Furthermore, throughout the section we assume that for every finite extension \(L/k\) and every \(\sigma \in \H^1(L, G)\), the Tate--Shafarevich group \(\Sha(A_L^\sigma)\) is finite.

For $\mathcal{B}\in \{\locconstBr_{Ma},
\locconstBr_{Cr}, \locconstBr_{Sk}, \Brnr\}$, let \(\RefinedObsDeg{U}{\delta}^{f, \mathcal{B}}\) be the refined obstruction set defined in~\eqref{eq:fBUobs}.  We first prove the following lemma, showing the existence of rational points on twists of base changes of $Y$ as a result of the non-emptiness of this obstruction set.

\begin{lemma}
\label{lemma:from-0-cycle-to-adelic-points}
Let $\delta \in \Z$, let $\mathcal{B}\in \{\locconstBr_{Ma},
\locconstBr_{Cr}, \locconstBr_{Sk}, \Brnr\}$ and suppose that \(\RefinedObsDeg{U}{\delta}^{f, \mathcal{B}} \neq \emptyset \). Then there exist finite extensions \(L_1, \ldots, L_r\) of \(k\) and \(\sigma_i \in \H^1(L_i, G) \) such that \(\gcd_{1 \le i \le r}([L_i:k])\) divides \(\delta\) and \(V_{L_i}^{\sigma_{i}}((L_{i})_{\Omega})^{\mathcal{B}} \neq~\emptyset\) for \(1\leq i\leq r\). Consequently, $Y_{L_i}^{\sigma_i}(L_i)\neq\emptyset$ for \(1\leq i\leq r\).
\end{lemma}

\begin{proof}
Let \(z \in \RefinedObsDeg{U}{\delta}^{f, \mathcal{B}}\). By Definitions~\ref{defn:refined-obstruction} and~\ref{defn:f-Brnr_pt}, there exist finite collections of: 
\begin{itemize}
    \item field extensions $L_1, L_2, \ldots ,L_r$,
    \item classes $\sigma_{i,1},\ldots, \sigma_{i,m_i}\in \H^1(L_i,G)$ for $1\leq i\leq r$, and
    \item  $\Omega_{L_i}$-local points $y_{i,j}\in V_{L_i}^{\sigma_{i,j}}((L_{i})_{\Omega})^{\mathcal{B}}$
\end{itemize}
such that $z$ is a $\Z$-linear combination of the $N_{L_i/k}^\Omega \bigl( f_{L_i}^{\sigma_{i,j}}(y_{i,j})\bigr)$, where $f_{L_i}^{\sigma_{i,j}}(y_{i,j})\in U_{L_i}((L_i)_{\Omega})$ is viewed as an element of $ Z_{0,\Omega}^1(U_{L_i})$. By~\eqref{eq:adelic cores}, $N_{L_i/k}^\Omega \bigl( f_{L_i}^{\sigma_{i,j}}(y_{i,j})\bigr)$ has degree $[L_i:k]$. Thus, $\delta=\deg(z)$ is a $\Z$-linear combination of the degrees $[L_i:k]$, which implies that \(\gcd_{1 \le i \le r}([L_i:k])\) divides \(\delta\). For each \(i\), we now choose one $\sigma_{i,j}$ and denote it by \(\sigma_i\), so we have \(V_{L_i}^{\sigma_{i}}((L_{i})_{\Omega})^{\mathcal{B}} \neq \emptyset\) for \(1\leq i\leq r\). 
    
Now let $j:V_{L_i}^{\sigma_i}\hookrightarrow Y_{L_i}^{\sigma_i}$ denote the natural inclusion. 
Then the pullback gives a natural map $j^*: \mathcal{B}(Y_{L_i}^{\sigma_i})\to \mathcal{B}(V_{L_i}^{\sigma_i})$. Thus, for $x\in V_{L_i}^{\sigma_{i}}((L_{i})_{\Omega})^{\mathcal{B}}$ and $\alpha\in \mathcal{B}(Y_{L_i}^{\sigma_i})$, we have 
\[\lr{ j(x), \alpha}_{BM}=\lr{ x, j^*(\alpha)}_{BM}=0.\]
Therefore, $j$ induces a natural inclusion
\[\emptyset\neq V_{L_i}^{\sigma_{i}}((L_{i})_{\Omega})^{\mathcal{B}}\subset Y_{L_i}^{\sigma_{i}}((L_{i})_{\Omega})^{\mathcal{B}}. \]

Recall that the action of $G$ on $Y$ is induced by the action on $A$. Thus, $Y_{L_i}^{\sigma_i}$ is an $L_i$-torsor under $A_{L_i}^{\sigma_i}$. Now, by~\eqref{eq: russian B sets} and Theorem~\ref{thm:HP/WA-torsors}\ref{part:HP/WA-torsors-2}, we have $Y_{L_i}^{\sigma_i}(L_i)\neq\emptyset$, as required.
\end{proof}

\subsection{Weak approximation}\label{sec:WAkummer}

We begin by proving an analogue of Theorem~\ref{thm:main intro-weak-approx} for the open subset $U$ defined in~\eqref{eq:Kummer}.

\begin{theorem}\label{thm:WAforU}
Let $A, U, G, f$ be as in \eqref{eq:Kummer}. Assume that for every finite extension \(L/k\) and every \(\sigma \in \H^1(L, G)\), the Tate--Shafarevich group \(\Sha(A_L^\sigma)\) is finite. Let $\delta\in\Z$. 
Then $Z_0^\delta(U)$ is dense in $\widetilde{Z}^\delta_{0,\Omega}(U)^{f,\Br_{\mathrm{nr}}}$ with respect to the weak approximation topology.
\end{theorem}

To prove this theorem, we use the following result on the density of rational points on twists of base changes of $V$.
\begin{lemma}
\label{lem:dense}
Let $L/k$ be a finite extension and let $\sigma\in \H^1(L,G)$.
Then $V_L^\sigma(L)$ is dense in $V_L^{\sigma}(L_{\Omega})^{\Brnr}$. 
In particular, if $V_L^{\sigma}(L_{\Omega})^{\Brnr} \neq \emptyset$ then $V_L^\sigma(L)\neq\emptyset$ and consequently $U(L)\neq \emptyset$.
\end{lemma}

\begin{proof}
Since the first statement follows automatically if \(V_L^{\sigma}(L_{\Omega})^{\Brnr}\) is empty, we may assume it is non-empty.
We have $\emptyset\neq V_L^{\sigma}(L_{\Omega})^{\Brnr} \subset Y_L^{\sigma}(L_{\Omega})^{\Br}$.
By Theorem \ref{thm:HP/WA-torsors}\ref{part:HP/WA-torsors-1}, $Y_L^{\sigma}(L)$ is dense in $Y_L^{\sigma}(L_{\Omega})^{\Br}$. Since $V_L^\sigma$ is open in $Y_L^\sigma$, a point in $Y_L^{\sigma}(L)$ that is sufficiently close to a point in $V_L^{\sigma}(L_{\Omega})^{\Brnr}$ will lie in $V_L^\sigma(L)$. 
The non-emptiness of $V_L^\sigma(L)$ now follows from $V_L^{\sigma}(L_{\Omega})^{\Brnr}\neq\emptyset$, and the map $f^\sigma_L:V_L^\sigma(L)\to U_L(L)$ gives $U(L)\neq\emptyset$.
\end{proof}

\begin{proof}[Proof of Theorem~\ref{thm:WAforU}]
Fix $n\in\Z_{>0}$ and a finite set of places $S\subset\Omega_k$. Let $z=(z_v)_v\in \RefinedObsDeg{U}{\delta}^{f, \Br_{\mathrm{nr}}}$. We seek $u\in Z_0^\delta(U)$ such that $u$ and $z_v$ have the same image in $h^0(U_{k_v})/n$ for all $v\in S$. The proof of Lemma~\ref{lemma:from-0-cycle-to-adelic-points} shows that $z$ is a $\Z$-linear combination of 0-cycles of the form \[N_{L_i/k}^\Omega 
\bigl( f_{L_i}^{\sigma_{i,j}}(y_{i,j})\bigr)\in Z_{0,\Omega}^{[L_i:k]}(U).\] Therefore, given a finite extension $L/k$ of degree $d$, a class $\sigma\in \H^1(L, G)$, and an $\Omega_L$-local point $(y_w)_w\in V_L^\sigma(L_\Omega)^{\Brnr}$,
it suffices to show the existence of $x\in Z_0^d(U)$ such that $x$ and $\sum_{w\mid v}N_{L_{w}/k_{v}}
(f_{L}^{\sigma}(y_{w}))$ have the same image in $h^0(U_{k_{v}})/n$ for all $v\in S$. By Proposition~\ref{prop:BB-Wittenberg-rephrasing-WA}, it is enough to show the existence of $x\in U(L)$ such that the localisations $x_w\in U(L_w)$
are sufficiently close to the points $f_{L}^{\sigma}(y_{w})$ 
for all $w\mid v$ for all $v\in S$. We can then view $x$ as a 0-cycle of degree $d$ on $U$, whose image in $Z_0^d(U_{k_v})$ is $\sum_{w\mid v}x_w$. 
The existence of $x$ follows from Lemma~\ref{lem:dense}: the continuity of $f_L^\sigma: V_L^\sigma\to U_L$ implies that $f^\sigma_L(V_L^\sigma(L))$ is dense in $f^\sigma_L(V_L^{\sigma}(L_{\Omega})^{\Brnr})$.  
\end{proof}

\begin{proof}[Proof of Theorem \ref{thm:main intro-weak-approx}]
Lemma~\ref{lem:globalinobsX} gives $\ol{Z_0^\delta(X)}\subset \ol{Z_0^\delta(U)} \subset
\overline{\widetilde{Z}^\delta_{0,\Omega}(U)^{f,\Brnr}}$.
We now prove the reverse inclusion $\overline{\widetilde{Z}^\delta_{0,\Omega}(U)^{f,\Brnr}}\subset \ol{Z_0^\delta(X)}$. Let $(z_v)_v\in \ol{\widetilde{Z}^\delta_{0,\Omega}(U)^{f,\Brnr}}$. 
Fix $n\in\Z_{>0}$ and a finite set $S\subset\Omega_k$. Then there exists $(u_v)_v \in \widetilde{Z}^\delta_{0,\Omega}(U)^{f,\Brnr}$ such that $\iota_*[u_v]=[z_v]\in h_0(X_{k_v})/n$ for all $v\in  S$. By Theorem~\ref{thm:WAforU}, there exists \(\globcycu \in Z_0^{\delta}(U)\) such that \(\globcycu\) and $u_v$ have the same image in $h_0(U_{k_v})/n$ for all $v \in S$. Take $\globcyc=\iota_*(\globcycu)\in Z^\delta_0(X)$. Then \(\globcyc\) and \(z_v\) have the same image in \(h_0(X_{k_v})/n\) for all \(v \in S\).
\end{proof}

\subsection{Hasse principle}
\label{subsec:hasse}

We continue to retain the notation and assumptions of Section \ref{section:kummer-descent}. 
In particular, we assume that $\Sha(A_L^\sigma)$ is finite for all finite extensions $L/k$ and all $\sigma \in \H^1(L, G)$.
From Lemma \ref{lem:globalinobsX} and \eqref{eq:subgroups of Brnr}
we have the following chain of inclusions:
\[
\overline{Z^\delta_{0}(X)}\subset \overline{\widetilde{Z}^\delta_{0,\Omega}(U)^{f,\Brnr}}\subset \overline{\widetilde{Z}^\delta_{0,\Omega}(U)^{f,\mathcal{B}}}\subset Z^\delta_{0,\Omega}(X).
\]
where $\mathcal{B}\in\{\locconstBr_{Ma},
\locconstBr_{Cr}, \locconstBr_{Sk}, \Brnr\}$. 
Theorem~\ref{thm:main intro-weak-approx} implies that $\overline{\widetilde{Z}^\delta_{0,\Omega}(U)^{f,\Brnr}}\neq \emptyset$ if and only if $Z^\delta_{0}(X)\neq \emptyset$. In this section we prove the stronger result that any
failure of the Hasse principle for 0-cycles of degree $\delta$ on $X$ is explained by emptiness of $\overline{\widetilde{Z}^\delta_{0,\Omega}(U)^{f,\mathcal{B}}}$, for \emph{any}  $\mathcal{B}\in\{\locconstBr_{Ma}, \locconstBr_{Cr}, \locconstBr_{Sk}, \Brnr\}$.
Our proof relies on Theorem~\ref{thm:intro-main-rank-jumps}, which is proved in Section~\ref{sec:rankjumps}.

\begin{theorem}
\label{thm:hasseprinciplebody}
Let $A, X, U, G, f$ be as in \eqref{eq:Kummer}. Assume that for every finite extension \(L/k\) and every \(\sigma \in \H^1(L, G)\), the Tate--Shafarevich group \(\Sha(A_L^\sigma)\) is finite. Let $\delta \in \Z$. Then 
\[Z_0^\delta(X)=\emptyset \iff Z_0^\delta(U)=\emptyset \iff \widetilde{Z}^\delta_{0,\Omega}(U)^{f,\mathcal{B}}=\emptyset\iff \overline{\widetilde{Z}^\delta_{0,\Omega}(U)^{f,\mathcal{B}}}=\emptyset.\] 
where $\mathcal{B}\in \{ \locconstBr_{Ma},\locconstBr_{Sk},\locconstBr_{Cr},\Brnr\}$.
\end{theorem}

Recall that $Y$ is a $k$-torsor under $A$. It follows that for $\sigma \in \H^1(L,G)$, the twist $Y^\sigma_L$ is an $L$-torsor under $A^\sigma_L$. Consequently, the Albanese variety of $Y^\sigma_L$ is $A^\sigma_L$; this is the dual of the Picard variety $\pic^0(Y^\sigma_L)$. Recall that an abelian variety is isogenous to its dual, and finiteness of Tate--Shafarevich groups is isogeny invariant.
Therefore, under the assumptions of Theorem~\ref{thm:hasseprinciplebody}, 
the groups $\locconstBr(Y^\sigma_L)/\Br_0(Y^\sigma_L)$ are finite by Remark~\ref{remark:locconstBr finite}, where $\locconstBr \in \{ \locconstBr_{Ma},\locconstBr_{Sk},\locconstBr_{Cr}\}$.

\begin{lemma}
\label{lem:L'}
Let $n\in\Z_{>0}$, let $L$ be a finite extension of $k$ and let $\sigma\in \H^1(L,G)$ be such that $Y_L^\sigma(L)\neq \emptyset$. Then there exists $L'/L$ finite with $[L':L]$ coprime to $n$ such that $V_L^\sigma(L')\neq\emptyset$ and therefore $U(L')\neq\emptyset$.
\end{lemma}

\begin{proof}
Since $Y_L^\sigma(L)\neq\emptyset$, the $L$-torsor $Y^\sigma_L$ is isomorphic to the underlying abelian variety $A_L^\sigma$. Choosing such an isomorphism, we identify $Y^\sigma_L$ with $A_L^\sigma$, and $\Fix_{G}(Y^\sigma_L)$ with $\Fix_{G}(A_L^\sigma)$, which is finite by assumption (see Section~\ref{section:kummer-descent}). By Theorem~\ref{thm:intro-main-rank-jumps} applied to $A_L^\sigma$, there exists $L'/L$ of degree $\ell$ for some prime $\ell>n$ such that $\rk A_L^\sigma(L')>\rk A_L^\sigma(L)$. In particular, there exists $P\in A_L^\sigma(L')\setminus \Fix_{G}(A_L^\sigma)$. This $P$ corresponds to a point in $Y_L^\sigma(L')\setminus \Fix_{G}(Y_L^\sigma)=V^\sigma_L(L')$. Applying $f_L^\sigma: V_L^{\sigma}\to U_L$ to this point completes the proof.
\end{proof}

\begin{proof}[Proof of Theorem~\ref{thm:hasseprinciplebody}]
Lemma~\ref{lem:globalinobsX} gives $Z_0^\delta(X)\subset \overline{Z_0^\delta(U)}\subset \ol{\widetilde{Z}^\delta_{0,\Omega}(U)^{f,\mathcal{B}}}$, which implies all three reverse implications. Henceforth, assume that $\overline{\widetilde{Z}^\delta_{0,\Omega}(U)^{f,\mathcal{B}}}\neq\emptyset$. By Lemma~\ref{lemma:from-0-cycle-to-adelic-points}, there exist finite extensions \(L_1, \ldots, L_r\) of \(k\) and \(\sigma_i \in \H^1(L_i, G) \) such that \(\gcd_{1 \le i \le r}([L_i:k])\) divides \(\delta\) and $Y_{L_i}^{\sigma_i}(L_i)\neq\emptyset$ for \(1\leq i\leq r\). For each $j\in \{1,\dots, r\}$, let $n_j=\gcd_{\substack{i\in\{1,\dots, r\}\setminus\{j\}}}([L_i:k])$. Note that if $r=1$ then, by convention, we take the gcd over an empty indexing set to be $1$. Apply Lemma~\ref{lem:L'} with $n=n_1$ to replace $L_1$ by a finite extension $L'_1$ such that $U(L_1')\neq\emptyset$ and 
\[\gcd\{[L_1':k], [L_2:k], \dots, [L_r:k]\}=\gcd_{1 \le i \le r}([L_i:k]).\]
Repeating this procedure for each \(j\), we obtain finite extensions $L_1',\dots, L_r'$ of $k$ such that $U(L_i')\neq\emptyset$ for $1\leq i\leq r$ and $\gcd_{1 \le i \le r}([L_i':k])=\gcd_{1 \le i \le r}([L_i:k])$. Since $\gcd_{1 \le i \le r}([L_i:k])$ divides $\delta$, the $L_i'$-points of $U$ can be used to construct a 0-cycle in $Z_0^\delta(U)\subset Z_0^\delta(X)$.
\end{proof}

\section{Rank jumps for abelian varieties}\label{sec:rankjumps}
In this section, we prove Theorem~\ref{thm:intro-main-rank-jumps}, restated below.

\rankjump*

In the case where $A$ is an elliptic curve, this statement is proven in \cite{daveacta}. The strategy underlying the proof of Theorem~\ref{thm:intro-main-rank-jumps} is to find a curve \(C \subset A\) containing the identity \(O_A\) for the group law on \(A\) that also
contains infinitely many degree \(d\) points defining infinitely many distinct degree $d$ extensions of $k$. The existence of a curve satisfying the former condition is guaranteed by a strong version of Bertini's theorem \cite{Diaz1991StrongBertini}.
To prove the latter condition in the case of a curve of genus at least two, we apply \cite[Proposition 3.3.1]{VirayVogt2024} and utilise it for the rank jump statement using Lemma~\ref{lemma:dave}, which appears in~\cite[Lemma 4.2.1]{Dave}.

\subsection{Bertini Theorems} 

Recall that a linear system $V$ on a nonsingular projective variety $X$ is a subspace of the space of global sections $\Gamma(X,\mathscr{L})$, where $\mathscr{L}$ is an invertible sheaf on $X$. By \cite[II, Proposition 7.7]{hartshorne}, members of a linear system $V$ correspond to effective divisors by taking their vanishing loci: we will make this identification and abuse language accordingly.
The points on $X$ on which all $s\in V$ vanish are called the base points of $V$.
The base locus is the scheme-theoretic intersection of the vanishing loci of all the $s\in V$. It forms a closed subscheme of $X$.

A useful tool for us will be a strong version of  Bertini's theorem, which asserts that a generic divisor in a linear system $V$ on $X$ inherits nice properties of $X$. We recall the classical version first.

\begin{theorem}[Bertini]
\label{thm:bertini}
\cite[Corollaire 6.11]{Jouanolou}
Let $k$ be a field of characteristic zero, and let $X$ be a smooth, geometrically integral $k$-variety. Let $\mathscr{L}$ be an invertible sheaf, and $V \subset \Gamma(X, \mathscr{L})$ a base-point-free linear system. 
Then a generic member of the linear system $V$ is smooth, geometrically reduced, and of dimension $\dim X-1$. 
\end{theorem}

Bertini's theorem can be generalised to assert smoothness away from the base locus of the linear system.

\begin{theorem}[Bertini with base locus]
\label{thm:bertinibaselocus}
Let $k$ be a field of characteristic zero, and let $X$ be a smooth, geometrically integral $k$-variety. Let $\mathscr{L}$ be an invertible sheaf, and $V \subset \Gamma(X, \mathscr{L})$ a linear system with base locus $B$. Then a generic member of the linear system $V$ is smooth away from $B$ and of dimension $\dim X-1$.
\end{theorem}

\begin{proof}
If $B \subset X$ is the base locus for \(V\), we can apply Theorem \ref{thm:bertini} to $X \setminus B$ and the linear system given by the image of \(V\) under the restriction map \(\Gamma(X, \mathscr{L}) \to \Gamma(X \setminus B, \mathscr{L}\vert_{X \setminus B})\).
Then a generic member of this restriction is smooth and of dimension $\dim X - 1$. This implies that a generic member of $V$ is also of dimension $\dim X - 1$ and smooth outside of $B$.
\end{proof}

The following result of D\'iaz and Harbater also guarantees smoothness at the base locus, provided the dimension of the base locus is not too large.

\begin{theorem}[{Strong Bertini, \cite[Corollary 2.4]{Diaz1991StrongBertini}}]
\label{thm:strongbertini}
Let $k$ be a field of characteristic zero. Let $V$ be a linear system on a smooth, projective $k$-variety $X$. Let $B$ be the base locus of $V$. 
If $B$ is reduced and nonsingular, and if the dimension of a component of $B$ is less than $\frac{1}{2}\dim X$, then a generic member of the linear system is nonsingular along this component. 
\end{theorem}

For our result, we will not only need the smoothness of a generic hyperplane section, but also geometric integrality.

\begin{lemma}\label{lem:bertiniwithpoint}
Let $k$ be a field of characteristic zero, and let $X$ be a smooth, projective, geometrically integral $k$-variety such that \(\dim(X) \ge 2\). Let $\mathscr{L}$ be an ample invertible sheaf, and $V \subset \Gamma(X, \mathscr{L})$ a linear system with base locus $B$. Suppose $B$ is reduced and nonsingular, and $\dim B < \frac{1}{2}\dim X$. Then a generic member of the linear system $V$ is smooth, projective, geometrically integral and of dimension $\dim X - 1$.
\end{lemma}

\begin{proof}
By Theorems~\ref{thm:bertinibaselocus} and~\ref{thm:strongbertini}, we know a generic member of $V$ is smooth and of dimension $\dim X - 1$. Since it is smooth, it is also geometrically reduced \cite[\href{https://stacks.math.columbia.edu/tag/056T}{Tag 056T}]{stacks-project}. Furthermore, since \(\dim(X) \ge 2\) and $\mathscr{L}$ is ample, such a member is geometrically connected by the Lemma of Enriques--Severi--Zariski \cite[III, Corollary 7.9]{hartshorne}. Since it is smooth and geometrically connected, we conclude that it is geometrically integral  \cite[\href{https://stacks.math.columbia.edu/tag/0BUG}{Tag 0BUG}]{stacks-project}.
\end{proof}

\subsection{Proof of Theorem \ref{thm:intro-main-rank-jumps}}
In this section, all algebraic field extensions of $k$ are taken inside a fixed algebraic closure $\ol{k}$ of $k$. The following key lemma is due to  Mendes da Costa. We include its proof for the reader's convenience. 

\begin{lemma}
\label{lemma:dave}
\cite[Lemma 4.2.1]{Dave}
Let $A$ be an abelian variety defined over a number field $k$, let $\ell\in\Z$ be a prime and suppose there are infinitely many degree $\ell$ extensions $L/k$ with $A(k)\subsetneq A(L)$.
Then
\begin{align*}
\# \{L/k \mid [L:k] = \ell \textrm{ and }\ \exists P \in A(L)\setminus A(k) \text{ such that } nP \in A(k) \textrm{ for some } n \in \Z_{>0}\} < C(A,\ell),
\end{align*}
where $C(A,\ell)$ is a constant depending only on $A$ and $\ell$. In particular, for all but finitely many degree $\ell$ extensions $L$ with $A(k)\subsetneq A(L)$, 
\[
\operatorname{rk} A(L) > \operatorname{rk} A(k).
\]
\end{lemma}

\begin{proof}
Let \[S=\{L/k \mid [L:k] = \ell \textrm{ and }\ \exists P \in A(L)\setminus A(k) \text{ such that } nP \in A(k) \textrm{ for some } n \in \Z_{>0}\}\] 
and let
\[S_{\mathrm{tors}} = \{L/k \mid [L \colon k] = \ell \textrm{ and } A(k)_{\mathrm{tors}}\subsetneq A(L)_{\mathrm{tors}}\}.\]

\textbf{Step 1.} First, we bound $S_{\mathrm{tors}}$ by a constant depending only on $A$ and $\ell$. 

Let $X>0$. Let $k_\ell$ be the compositum inside $\ol{k}$ of all degree $\ell$ extensions of $k$ whose discriminant is at most $X$, and let $M_\ell$ denote the Galois closure of $k_\ell/k$. By Hermite's theorem, \(M_\ell\) is a finite extension of \(k\).
Observe that the exponent of $\Gal(M_\ell/k)$ divides $\ell!$. This is because $M_\ell$ is the compositum of the Galois closures of all degree $\ell$ extensions of $k$ with discriminant at most $X$, so $\Gal(M_\ell/k)$ is a subgroup of the direct product of groups $\Gal(L/k)$, where $L$ runs over the Galois closures of those degree $\ell$ extensions. For each such $L$, we have $\#\Gal(L/k) \mid \ell!$, so the exponent of $\Gal(M_\ell/k)$ divides $\ell!$, as claimed.

Let $\cO_\ell$ be the ring of integers of $M_\ell$, let $\kappa_{\mathfrak{p}}$ denote the residue field at a prime $\mathfrak{p}\subset \cO_k$, and $\kappa_\mathfrak{q}$ the residue field at a prime $\mathfrak{q}\subset \cO_\ell$ above $\mathfrak{p}$. Recall that $\Gal(\kappa_\mathfrak{q}/\kappa_\mathfrak{p})\cong D_\mathfrak{q}/I_\mathfrak{q}$ is a cyclic subquotient of $\Gal(M_\ell/k)$. Therefore, $\# \Gal(\kappa_\mathfrak{q}/\kappa_\mathfrak{p}) \mid \ell!$ and consequently $\#\kappa_\mathfrak{q}\leq \#\kappa_\mathfrak{p}^{\ell!}$. 

Let $\mathfrak{p}\subset \cO_k$ be a prime of good reduction of \(A\) lying above a rational prime \(p \in \Z\).
Let $A(M_\ell)(p')$ denote the set of torsion points in $A(M_\ell)$ of order coprime to $p$. 
Then for any prime $\mathfrak{q} \subset \cO_\ell$ above $\mathfrak{p}$, there is an embedding \cite[Theorem C.1.4]{HindrySilverman}
\[
A(M_\ell)(p')\hookrightarrow A(\kappa_{\mathfrak{q}}).
\]

Let $\mathfrak{p}_1,\mathfrak{p}_2$ be the two smallest primes of good reduction of $A$ lying above distinct rational primes, and let $\mathfrak{q}_1,\mathfrak{q}_2\in \mathcal{O}_\ell$ be primes lying above $\mathfrak{p}_1,\mathfrak{p}_2$ respectively. 
Using the embeddings above, we get \[\#A(M_\ell)_{\mathrm{tors}} \leq \# A(\kappa_{\mathfrak{q}_1}) \# A(\kappa_{\mathfrak{q}_2}).\] We can bound $\#A(\kappa_{\mathfrak{q}_i})$ in terms of $\#\kappa_{\mathfrak{q}_i}$ and $\dim(A)$. Since $\#\kappa_{\mathfrak{q}_i} \leq \#\kappa_{\mathfrak{p}_i}^{\ell!}$, this means there is a bound $\#A(M_\ell)_{\mathrm{tors}} \leq B(A, \ell)$ depending only on $A$ and $\ell$, independent of $X$.
Letting $X\rightarrow \infty$ we see that 
\[\# S_{\mathrm{tors}}\leq B(A, \ell).\]

\textbf{Step 2.} 
Let $L \in S \setminus S_{\mathrm{tors}}$, and let $P \in A(L)\setminus A(k)$. We show that if $n\in\Z$ is coprime to $\ell!$ then $n P \notin A(k)$.

Let $n$ be coprime to $\ell!$ and suppose for contradiction that $n P=Q$ for some $Q\in A(k)$. 
Let $M$ be the Galois closure of $L/k$. 
Consider the following commutative diagram, where the vertical
exact sequence is the inflation-restriction sequence. 

\begin{center}
\begin{tikzcd}
            &                                          &  & 0 \arrow[d]                                  &  &                                        &   \\
            &                                          &  & {\H^1(\Gal(M/k),A[n](M))} \arrow[d]        &  &                                        &   \\
0 \arrow[r] & A(k)/n A(k) \arrow[rr] \arrow[d, "f"] &  & {\H^1(k,A[n])} \arrow[rr] \arrow[d, "g"] &  & {\H^1(k,A)[n]} \arrow[r] \arrow[d] & 0 \\
0 \arrow[r] & A(M)/n A(M) \arrow[rr]                &  & {\H^1(M,A[n])} \arrow[rr] &  & {\H^1(M,A)[n]} \arrow[r]           & 0
\end{tikzcd}
\end{center}
Since $\Gal(M/k)$ has order dividing $\ell!$ and $n$ is coprime to $\ell!$, we have $\H^1(\Gal(M/k),A[n](M))=0$ by Schur's Lemma and hence $g$ is injective. A diagram chase implies that $f$ is injective and therefore $Q=nR$ for some $R\in A(k)$. This implies that $n(P - R) = O_A$, and hence $P - R \in A(L)_{\mathrm{tors}}\setminus A(k)$. In particular, $L \in S_{\mathrm{tors}}$, giving the desired contradiction.

\textbf{Step 3.} For each $m\in \Z_{>0}$, we show that the set
\[S_m \coloneq\{L/k \mid [L:k] = \ell,\ \exists P \in A(L)\setminus A(k) \text{ such that } m P \in A(k)\}\]
is finite, and its size can be bounded purely in terms of $A$ and $m$.

Let $\{P_1, \ldots, P_r\}$ be a set of generators for the finitely generated group $A(k)$. Consider the finite subset $A_{\mathrm{fin}}(k) = \{a_1 P_1 + \cdots + a_r P_r \mid 0 \leq a_i < m\} \subset A(k)$. Let $L \in S_{m}$ and let $P \in A(L)\setminus A(k)$ be such that $m P \in A(k)$.
Then we can write $m P = R + m Q$ for some $R \in A_{\mathrm{fin}}(k)$ and $Q \in A(k)$, so that have $P-Q \in A(L)\setminus A(k)$ and $m(P-Q) \in A_{\mathrm{fin}}(k)$. Since $A_{\mathrm{fin}}(k)$ and $A[m]$ are finite, there are only finitely many choices for the point $P-Q$, and these finitely many points have finitely many fields of definition. Therefore, $S_m$ is finite and its size can be bounded purely in terms of $A$ and $m$, as claimed.

\textbf{Step 4.} We show that \[S=S_{\mathrm{tors}} \cup \bigcup_{m \leq \ell} S_{m}.\] 

Let $L \in S \setminus (S_{\mathrm{tors}} \cup \bigcup_{m \leq \ell} S_{m})$ and let $n_0$ be the smallest $n\in \Z_{>0}$ such that there exists $P\in A(L)\setminus A(k)$ with $nP\in A(k)$. Since $L\notin S_{\mathrm{tors}}$, Step 2 shows that $n_0$ is not coprime to $\ell!$ so we can write $n_0=rq$ for some $1 < q \leq \ell$. Since $L\notin S_{q}$, we have $qP\in A(L)\setminus A(k)$ and $rqP=n_0 P\in A(k)$. But this contradicts the minimality of $n_0$.

Now our proof is complete, since $\# S_{\mathrm{tors}}$ and $\#\bigcup_{m \leq \ell} S_{m}$ can be bounded purely in terms of $A$ and $\ell$.
\end{proof}

Let $\iota:A \hookrightarrow \mathbb{P}_k^n$ be an embedding of $A$ into projective space of dimension $n$. Consider the line bundle $\mathcal{O}_{\mathbb{P}^n_k}(1)$. 
The space of global sections \(\Gamma(\mathbb{P}^n_k,\mathcal{O}_{\mathbb{P}^n_k}(1))\) is generated by the sections $x_0,\dots, x_n$,
corresponding to the projective coordinates on $\mathbb{P}^n_k$.
Let $O_A$ be the identity element of $A$, and let  $O\coloneq\iota(O_A)$. The hyperplanes passing through $O = [o_0:\dots:o_n]$ correspond to a linear subspace $V \subset \Gamma(\mathbb{P}^n_k,\mathcal{O}_{\mathbb{P}^n_k}(1))$ of codimension 1. More precisely, the hyperplane $a_0x_0+a_1x_1+\dots +a_nx_n\in \Gamma(\mathbb{P}^n_k,\mathcal{O}_{\mathbb{P}^n_k}(1))$ is in $V$ if
\[
a_0o_0+a_1o_1+\dots +a_no_n=0.
\]
This linear subspace $V$ is a linear system. The base locus of this linear system is the scheme-theoretic intersection of the vanishing loci of the sections of $V$, which is $O$. Pulling \(V\) back along \(\iota \colon A \to \P^n\), we get a linear system on \(A\) whose base locus is \(O_A\).

\begin{proof}[Proof of Theorem~\ref{thm:intro-main-rank-jumps}]
Let $g$ denote the dimension of the abelian variety $A$. In the case $g=1$, Theorem~\ref{thm:intro-main-rank-jumps} (with $N=2$) follows from~\cite[Corollary~2]{daveacta}. 
Henceforth, suppose $g\geq 2$. 
Let $\iota:A\rightarrow \mathbb{P}^n_k$ be an embedding of $A$ into projective space. Such an embedding always exists if $n\geq 2g+1$. Now let $\iota^*V\subset \Gamma(A, \iota^*\mathcal{O}_{\mathbb{P}^n_k}(1))$ be the restriction to $A$ of the linear system of hyperplanes passing through $O$, as constructed above. The base locus $B$ of this linear system is the point $O_A$. Since $\dim \{O_A\}=0<\frac{1}{2}g$, Lemma~\ref{lem:bertiniwithpoint} allows us to 
choose a member of $\iota^*V$ (corresponding to a hyperplane section of the image \(\iota(A)\)) that is smooth, geometrically integral and of dimension $g-1$. 
Note that \(O_A\) lies on this member.
Iterating this procedure we obtain a smooth, geometrically integral curve $C$ with $O_A\in C\subset A$.

Let $g(C)$ denote the genus of $C$. Because an abelian variety cannot contain a rational curve (see e.g.\ \cite[Corollary~3.8]{milneAV}), we have $g(C) > 0$. First, suppose that $g(C) = 1$. Then $C$ is an elliptic curve, because $O_A \in C(k)$. Furthermore, the inclusion \(C \hookrightarrow A\) necessarily factors through the Albanese map \(C \to \mathrm{Jac}(C) \to A\), so the group  structure on the points of \(C \cong \mathrm{Jac}(C)\) is compatible with that on \(A\). Now~\cite[Corollary~2]{daveacta} shows that for every prime $\ell$, there exist infinitely many extensions $L/k$ of degree $\ell$ such that $\rk C(L)>\rk C(k)$. Take one such extension $L/k$ and suppose for contradiction that $\rk A(L)=\rk A(k)$. Then the group index $[A(L):A(k)]$ is finite and therefore $[C(L):C(k)]$ is finite, which gives the desired contradiction with the fact that $\rk C(L)>\rk C(k)$. 

Suppose $g(C) > 1$. Let $N = 2 g(C)$. By Riemann--Roch, for any prime $\ell \geq N$,
the divisor $\ell O_A$ defines a linear system with no base points, and so defines a morphism $C\rightarrow \mathbb{P}^1_k$ of degree $\ell$ (using \cite[Lemma 3.1.2]{VirayVogt2024}). Now \cite[Proposition 3.3.1]{VirayVogt2024}
shows that $C$ has infinitely many points of degree $\ell$. Since $C(L)$ is finite for any finite extension $L/k$, 
we conclude that there are infinitely many field extensions $L/k$ of degree $\ell$ for which $A(k)\subsetneq A(L)$. 
An application of Lemma~\ref{lemma:dave} completes our proof in this case.
\end{proof}

\section{Bielliptic surfaces and friends: proof of Theorem~\ref{thm:intro bielliptic friend}}
\label{sec:quotients-general}

We now describe a more general set-up for which the methods in Section \ref{sec:weak-approx} apply quite directly, which includes the case of bielliptic surfaces. 

Let \(G\) be a linear algebraic group over \(k\) and let \(X\) be a smooth, proper, geometrically integral variety over \(k\) such that \(X\) admits a morphism \(f \colon Y \to X\) giving \(Y\) the structure of an \(X\)-torsor under \(G\). Suppose further that \(Y\) is a \(k\)-torsor under \(A\) for an abelian variety \(A\). Analogously to Section \ref{sec:Xobs}, for each \(\mathcal{B} \in  \{ \locconstBr_{Ma},\locconstBr_{Sk},\locconstBr_{Cr},\Br\}\), we use Definition \ref{defn:f-Brnr_pt} to define the obstruction set
\begin{equation*}\label{eq:fBXobs} \widetilde{Z}_{0,\Omega}(X)^{f,\mathcal{B}}\coloneq\refrec_{X, \Omega}\left(\bigoplus_{L/k \ \textrm{finite}}\Z[X(L_{\Omega})^{f,\mathcal{B}}]\right)  \subset Z_{0,\Omega}(X).\
\end{equation*}
The following theorem is a corollary of Proposition \ref{prop:propofrefobs} and the proof of Theorem~\ref{thm:hasseprinciplebody}.

\begin{theorem} \label{thm:quotabvars}
Maintaining the above notation, suppose that for every finite extension \(L/k\) and every \(\sigma \in \H^1(L, G)\), we have that \(\Sha(A^{\sigma}_L)\) is finite. Let $\delta \in \Z$. Then
\begin{enumerate}[(i)]
    \item the set \(Z_0^{\delta}(X)\) is dense in \(\RefinedObsDeg{X}{\delta}^{f,\Br}\) with respect to the weak approximation topology;
    \item for any \(\mathcal{B} \in  \{ \locconstBr_{Ma},\locconstBr_{Sk},\locconstBr_{Cr},\Br\}\), if \(\RefinedObsDeg{X}{\delta}^{f,\mathcal{B}} \neq \emptyset\), then \(Z_0^{\delta}(X) \neq \emptyset\).
\end{enumerate}
\end{theorem}

Of particular interest is the application to bielliptic surfaces. Recall from the Enriques--Kodaira classification for surfaces that a bielliptic surface $X$ over a number field $k$ is a smooth, projective, geometrically integral surface with Kodaira dimension \(0\), geometric genus $p_g(X) = 0$ and irregularity $q(X) = 1$. This implies that $\chi(\mathcal{O}_X) = 0$ (using the formula $\chi(\mathcal{O}_V) = 1 - q(V) + p_g(V)$ for smooth projective surfaces $V$ of Kodaira dimension \(0\)). Since $\omega_X$ is non-trivial, and either $\omega_X^{\otimes 4}$ or $\omega_X^{\otimes 6}$ is trivial
(\cite[Corollary VIII.7]{Beauville96-CAS}), it follows that $\omega_X$ yields a non-trivial canonical $X$-torsor $f : Y \to X$ under a finite linear algebraic group $F$ (of size equal to the order of $\omega_X$) over $k$. 
Moreover, $\omega_Y = \mathcal{O}_Y$ so that $p_g(Y)=1$ and $\chi(\mathcal{O}_Y) = \deg(f)\chi(\mathcal{O}_X) = 0$ (\cite[Lemma VI.3]{Beauville96-CAS}).
We deduce that $q(Y) = 2$ and thus that $Y$ is a $k$-torsor 
under its Albanese variety (i.e.\ $[Y] \in \H^1(k, A)$ for the abelian variety $A = \Alb_{Y/k}$).
Therefore, we are in the scenario of Theorem~\ref{thm:quotabvars} and can deduce the following corollary.

\begin{corollary}\label{cor:biell}
Let $X$ be a bielliptic surface over a number field $k$ together with its canonical torsor $f: Y \to X$ under a finite group $F$ with $Y$ a $k$-torsor under its Albanese variety $A = \Alb_{Y/k}$. Suppose that for every finite extension \(L/k\) and every \(\sigma \in \H^1(L, F)\), we have that \(\Sha(A^{\sigma}_L)\) is finite. Let $\delta \in \Z$.
Then
\begin{enumerate}[(i)]
    \item 
    the set \(Z_0^{\delta}(X)\) is dense in \(\RefinedObsDeg{X}{\delta}^{f,\Br}\) with respect to the weak approximation topology;\label{cor:biell(i)}
    \item for any \(\mathcal{B} \in  \{ \locconstBr_{Ma},\locconstBr_{Sk},\locconstBr_{Cr},\Br\}\), if \(\RefinedObsDeg{X}{\delta}^{f,\mathcal{B}} \neq \emptyset\), then \(Z_0^{\delta}(X) \neq \emptyset\).
\end{enumerate}
\end{corollary}

\begin{remark}
In \cite[Corollary 3.1]{Sko-Descent-EtBrauer}, under the usual finiteness assumption for Tate--Shafarevich groups of abelian varieties, Skorobogatov proves that if $X$ is a bielliptic surface over a number field $k$, then $X(k)$ is dense in the finite descent set $X(\A_k)^{\et}_\bullet$, where $_\bullet$ here denotes that, for any infinite place $v \in \Omega_k$, instead of considering the local points $X(k_v)$, we consider the connected components of $X(k_v)$, i.e.\ $\pi_0(X(k_v))$. Once one makes sense of this connected components subtlety in the context of 0-cycles, it is likely that one could use Proposition \ref{prop:propofrefobs} to prove a similar result to Corollary~\ref{cor:biell}\ref{cor:biell(i)}, but with $\widetilde{Z}^{\delta}_{0,\Omega}(X)^{f, \Br}$ replaced by $\widetilde{Z}^{\delta}_{0,\Omega}(X)^{\et}_{\bullet}$.
\end{remark}

\section{On a question of Zhang} 
\label{sec:huizhang}
In this section, we give another application of our refined obstruction sets. 
Let \(X\) be any smooth, proper, geometrically integral variety over a number field \(k\), and let
\[
Z_{0, \Omega}(X)^{\conn} \coloneq\bigcap_{\substack{G \text{ connected}\\ \text{over $k$}}} \bigcap_{[f \colon Y \to X] \in \H^1(X, G)} Z_{0, \Omega}(X)^{f},
\]
where \(Z_{0, \Omega}(X)^{f}\) is as in Definition \ref{defn:non-refined-f-descent}, \(G\) ranges over all connected linear algebraic groups over \(k\), and $f \colon Y \to X$ ranges over all $X$-torsors under $G$. 
In \cite[Remark 5.8]{Zhang-0cycles-25}, Zhang asks whether there are any reasonable assumptions under which $\overline{Z_{0, \Omega}(X)^{\conn}} = Z_{0, \Omega}(X)^{\Br}$, 
where the closure is with respect to the weak approximation topology. 
Using our refined obstruction sets, we are able to show that this is indeed the case under some conditions. We will prove Theorem~\ref{thm:conndense}, restated below.

\conndense*

The strategy for the proof of Theorem \ref{thm:conndense} is to show that, under the assumptions in the statement of the theorem, the closure of the refined Brauer--Manin obstruction set $\RefinedObs{X}^{\Br}$ is $Z_{0, \Omega}(X)^{\Br}$. Since the refined Brauer--Manin obstruction is defined using obstruction sets of the form $X(\A_L)^{\Br}$ for rational points, we can use a result of Harari (\cite[Th\'{e}or\`{e}me 2]{Harari02}, using the fact that \(\PGL_n\) is connected)
which shows that
$X(\A_L)^{\Br} = X(\A_L)^{\conn}$, where the latter is defined in~\eqref{eq:conn desc pts}. In particular,  we identify the refined Brauer--Manin obstruction \(\RefinedObs{X}^{\Br}\) with the refined connected descent obstruction set \[\RefinedObs{X}^{\conn} \coloneq \refrec_{X, \Omega}\left(\bigoplus_{L/k \ \textrm{finite}}\Z[X(\A_L)^{\conn}]\right)  .\]
Finally, we will show that $\RefinedObs{X}^{\conn} \subset Z_{0, \Omega}(X)^{\conn}$, 
which, together with $Z_{0, \Omega}(X)^{\conn} \subset Z_{0,\Omega}(X)^{\Br}$ (see \cite[Remark 5.8]{Zhang-0cycles-25}), yields the chain of inclusions
\[
\RefinedObs{X}^{\Br} = \RefinedObs{X}^{\conn} \subset Z_{0, \Omega}(X)^{\conn} \subset Z_{0,\Omega}(X)^{\Br}
\]
and thus the required result, since, as we will see, $\overline{\RefinedObs{X}^{\Br}} = Z_{0,\Omega}(X)^{\Br}$.
Let us now prove everything that is needed in the above strategy.

\begin{proposition} \label{prop:wudhiq}
Let $X$ be a smooth, proper, geometrically integral variety over a number field $k$, such that
\begin{itemize}
    \item[(i)] $Z_{0, \Omega}(X)^{\Br}$ is open in $Z_{0, \Omega}(X)$ with respect to the weak approximation topology, and 
    \item[(ii)] for every positive integer $d$, there exists a finite extension $k'/k$ such that for any degree $d$ extension
$K/k$ which is linearly disjoint from $k'/k$, the restriction map
\[
    \Br(X)/\Br_0(X) \to \Br(X_K)/\Br_0(X_K)
\]
is surjective.
\end{itemize}
Then $\overline{\RefinedObs{X}^{\Br}} = Z_{0,\Omega}(X)^{\Br}$.
\end{proposition}

\begin{proof}
By definition, we have $\RefinedObs{X}^{\Br} \subset Z_{0,\Omega}(X)^{\Br}$. Since $Z_{0,\Omega}(X)^{\Br}$ is closed in $Z_{0, \Omega}(X)$ (\cite[Proposition 5.2]{Zhang-0cycles-25}) we have $\overline{\RefinedObs{X}^{\Br}} \subset Z_{0,\Omega}(X)^{\Br}$.
We now show the other containment, using elements from the proofs of \cite[Theorem 5.6]{Zhang-0cycles-25} and \cite[Theorem 3.2.1]{Liang-Arithmetic-13}. Let $(z_v)_v \in Z_{0,\Omega}(X)^{\Br}$ be an $\Omega$-local 0-cycle of degree $\delta$. By assumption $(i)$, $Z_{0,\Omega}(X)^{\Br}$ is open, so we can find some positive integer $n_0$ and some finite set of places $S_0 \subset \Omega$ such that there is an open set $\mathcal{A}_{n_0, S_0} \subset Z_{0,\Omega}(X)^{\Br}$ (see Definition~\ref{def:weakapproxtop})\footnote{In fact, since \(Z_{0,\Omega}(X)^{\Br}\) is a group, it contains all translates of the form \((z'_v)_v + \mathcal{A}_{n_0,S_0}\) for \((z'_v)_v \in Z_{0,\Omega}(X)^{\Br}\).}. 
Following the proof of \cite[Theorem 5.6]{Zhang-0cycles-25}, it is sufficient to show that, for any positive multiple $n$ of $n_0$ and any finite set of places $S \supset S_0$, there exists some $y_{n,S} \in \RefinedObs{X}^{\Br}$ with $(z_v)_v - y_{n, S} \in \mathcal{A}_{n, S}$.
Fix such $n$ and $S$. Fix a closed point $\tilde{x} \in X$, say of degree $d_{\tilde{x}} \coloneq[k(\tilde{x}):k]$. Enlarging $S$ if necessary and using the Lang--Weil estimates, we can assume that $S$ contains all infinite places and that $X(k_v) \neq \emptyset$ for all $v \notin S$.

Using the trivial fibration $\pi: X \times \P^1_k \to \P^1_k$ as in the proofs of Liang and Zhang,
by \cite[Lemma 3.2]{CT00-Surfaces} 
we obtain a positive integer $r$ (depending only on $S$ and $(z_v)_v$) such that, after fixing $0 \in \P^1_k(k)$ and an isomorphism $\A^1_k \cong \P^1_k \setminus \{0\}$, for each $v \in S$ we can replace  $(z_v \times 0) + r(\tilde{x} \times 0)$ by a rationally equivalent 0-cycle $z_v^1$ which is effective, sufficiently close to $(z_v \times 0) + r(\tilde{x} \times 0)$ (implying that they have the same image in the Chow group modulo $n$), such that \(\pi_{*}(z_v^1)\) is separable (i.e.\ a sum of distinct closed points) and supported on $\A^1_k$, and such that  $d_r\coloneq\delta + r d_{\tilde{x}} > 0$
(independent of $v$). By assumption $(ii)$, we may fix a finite extension $k'/k$ such that, for any extension $L/k$ of degree $d_r$ linearly disjoint from $k'$, the restriction map $\Br(X)/\Br_0(X) \to \Br(X_L)/\Br_0(X_L)$ is surjective. The base change morphism \(\P^1_{k'} \to \P^1_{k}\) is \'{e}tale and thus defines a Hilbertian subset $\text{\textbf{Hil}} \subset (\P^1_k)_0$ (\cite[Definition 4.6]{Liang-compatibility-23}).
Then by the arguments in \cite[Theorem 5.6]{Zhang-0cycles-25} and \cite[Proposition 4.7]{Liang-compatibility-23}, which use Hilbert irreducibility (\cite{Ekedahl90}) and the implicit function theorem on $\sym^{d_r}(\P^1_{k_v})$ and $\sym^{d_r}(X_{k_v})$, there is a closed point $\theta \in \text{\textbf{Hil}}$
such that 
\begin{itemize}
    \item \(K \coloneq k(\theta)\) has degree \(d_r\) and is linearly disjoint from \(k'\) over \(k\); and
    \item there is a point $(x_w)_{w \in \Omega_K} \in X(\A_K)$ with
    \[ (z_v)_v + r \tilde{x} - N_{K/k}^\Omega
    ((x_w)_w) \in \mathcal{A}_{n, S}.\]
\end{itemize}

We now claim that $y_{n, S}\coloneq- r \tilde{x} + N_{K/k}^\Omega((x_w)_w) \in \RefinedObs{X}^{\Br}$.
Note that $N_{K/k}^\Omega((x_w)_w) \in Z_{0, \Omega}(X)^{\Br}$: indeed, $N_{K/k}^\Omega((x_w)_w) = (z_v)_v + r \tilde{x} + (a_v)_v$ for some $(a_v)_v \in \mathcal{A}_{n, S}\subset \mathcal{A}_{n_0, S_0} \subset Z_{0, \Omega}(X)^{\Br}$. Moreover, $(z_v)_v \in  Z_{0, \Omega}(X)^{\Br}$, and $r \tilde{x} \in Z_{0, \Omega}(X)^{\Br}$ since $\tilde{x}$ is a global point. Since $Z_{0, \Omega}(X)^{\Br}$ is a group, it follows that  $N_{K/k}^\Omega((x_w)_w) \in Z_{0, \Omega}(X)^{\Br}$.

Furthermore, $(x_w)_w \in X(\A_K)^{\Br}$: indeed, by assumption $(ii)$, we can write any $\gamma \in \Br(X_K)/\Br_0(X_K)$ as $\gamma = \alpha_K + \beta$, where $\alpha_K \coloneq \alpha \otimes_k K$ for some $\alpha \in \Br(X)$ and $\beta \in \Br_0(X_K)$. By standard properties of the Brauer--Manin pairing we have 
\[ \langle  (x_w)_w, \gamma \rangle_{BM} = \langle  (x_w)_w, \alpha_K\rangle_{BM} = \langle  N_{K/k}^\Omega((x_w)_w), \alpha \rangle_{BM} = 0,\]
since $N_{K/k}^\Omega((x_w)_w) \in Z_{0, \Omega}(X)^{\Br}$. Hence,  $(x_w)_w \in X(\A_K)^{\Br}$ and thus, by definition, $N_{K/k}^\Omega((x_w)_w)  \in \RefinedObs{X}^{\Br}$.

Finally, since $Z_0(X) \subset \RefinedObs{X}^{\Br}$, we have $r \tilde{x} \in  \widetilde{Z}_{0, \Omega}(X)^{\Br}$, and since $\RefinedObs{X}^{\Br}$ is a group, we have $y_{n, S} = - r \tilde{x} + N_{K/k}^\Omega((x_w)_w) \in  \RefinedObs{X}^{\Br}$, as required. 
\end{proof}

\begin{lemma} 
\label{lem:harconn}
Let $X$ be a smooth, proper, geometrically integral variety over a number field $k$. Then $\RefinedObs{X}^{\Br} = \RefinedObs{X}^{\conn}$.
\end{lemma}

\begin{proof} 
Under our assumptions on $X$, we have by \cite[Th\'{e}or\`{e}me 2]{Harari02} that $X(\A_L)^{\Br} = X(\A_L)^{\conn}$, for any finite extension $L/k$. Hence, we immediately conclude from the definitions that
\begin{align*}
\RefinedObs{X}^{\Br} &= \refrec_{X, \Omega}\left(\bigoplus_{L/k \ \textrm{finite}}\Z[X(\A_L)^{\Br}]\right)    \\ &=  \refrec_{X, \Omega}\left(\bigoplus_{L/k \ \textrm{finite}}\Z[X(\A_L)^{\conn}]\right)  \\ &= \RefinedObs{X}^{\conn}. \qedhere
\end{align*}
\end{proof}

\begin{lemma} 
\label{lem:1}
Let $X$ be a smooth, quasi-projective, geometrically integral variety over a number field $k$. Then $\RefinedObs{X}^{\conn} \subset Z_{0, \Omega}(X)^{\conn}$.
\end{lemma}

\begin{proof}
Let $z \in \RefinedObs{X}^{\conn}$. By the definition of the refined descent set, we can write 
\[ z = \sum_{i=1}^r \sum_{j_i= 1}^{r_i} n_{i, j_i} N^{\Omega}_{L_i/k}
(x_{i, j_i}),\]
for some finite extensions $L_i/k$, some integers $n_{i, j_i}$, and some adelic points $x_{i, j_i} \in X(\A_{L_i})^{\conn}$. To show that $z \in Z_{0, \Omega}(X)^{\conn}$, it suffices to show that, for any connected linear algebraic group $G$ over $k$ and for any $X$-torsor $f: Y \to X$ under $G$, we have $z \in Z_{0, \Omega}(X)^f$. Moreover, since each $Z_{0, \Omega}(X)^f$ is an abelian group, it suffices to show that $N^{\Omega}_{L_i/k}(x_{i, j_i}) \in Z_{0, \Omega}(X)^f$ for all $X$-torsors $f$ under connected linear algebraic groups over $k$ and all $x_{i, j_i} \in X(\A_{L_i})^{\conn}$. So fix an $X$-torsor $f: Y \to X$ under a connected linear algebraic group $G$ and fix $x_{i, j_i} \in X(\A_{L_i})^{\conn}$. By definition of $X(\A_{L_i})^{\conn}$, there exists some $\sigma \in H^1(L_i, G)$ such that $x_{i, j_i} \in f_{L_i}^{\sigma}(Y_{L_i}^{\sigma}(\A_{L_{i}}))$, say $x_{i, j_i} = f_{L_i}^{\sigma}((y_w)_w)$ with $(y_w)_w \in Y_{L_i}^{\sigma}(\A_{L_{i}})$, where we have used the fact that the base-change $f_{L_i}: Y_{L_i} \to X_{L_i}$ is a torsor under $G_{L_i}$, with $G_{L_i}$ still a connected linear algebraic group. Hence, viewing $x_{i, j_i}$ and $(y_w)_w$ as $\Omega_{L_i}$-local 0-cycles of degree 1 in $Z_{0, \Omega}(X_{L_i})$ and $Z_{0, \Omega}(Y^{\sigma}_{L_i})$, respectively, we have \[N^{\Omega}_{L_i/k}(x_{i, j_i}) = N^{\Omega}_{L_i/k}(f_{L_i, \ast}^{\sigma}((y_w)_w)) = \rec_f((y_w)_w).\] By definition of $Z_{0, \Omega}(X)^f$, this implies that $N^{\Omega}_{L_i/k}(x_{i, j_i}) \in Z_{0, \Omega}(X)^f$, as required.
\end{proof}

\begin{remark}\label{rem:1}
In the proof of Lemma \ref{lem:1}, we only use the fact that connected linear algebraic groups remain so after base extensions. This means that if we consider any collection $\mathscr{G}$ of linear algebraic groups over $k$ that is similarly `closed under base extensions', a similar proof shows that $\widetilde{Z}_{0, \Omega}(X)^{\mathscr{G}} \subset Z_{0, \Omega}(X)^{\mathscr{G}}$. Examples of families $\mathscr{G}$ satisfying this property include all linear algebraic groups, all connected linear algebraic groups (as in Lemma \ref{lem:1}), all commutative linear algebraic groups, all tori, all groups of multiplicative type, all solvable linear algebraic groups, all finite linear algebraic groups, and many other important families.
\end{remark}

\begin{proof}[Proof of Theorem \ref{thm:conndense}]
By Lemmas \ref{lem:harconn} and \ref{lem:1}  above, together with \cite[Remark 5.8]{Zhang-0cycles-25}), we have the chain of inclusions
\[
\RefinedObs{X}^{\Br} = \RefinedObs{X}^{\conn} \subset Z_{0, \Omega}(X)^{\conn} \subset Z_{0,\Omega}(X)^{\Br}
\]
and thus the result follows from Proposition~\ref{prop:wudhiq}.
\end{proof}

We now show that Theorem~\ref{thm:conndense} applies when $X$ is either rationally connected or a K3 surface.

\begin{proposition}
\label{prop:notvacuous}
Conditions~\ref{part:conndense-cond1} and~\ref{part:conndense-cond2} of Theorem~\ref{thm:conndense} are satisfied for a smooth, proper, geometrically integral variety \(X\) over a number field $k$ if:
\begin{enumerate}[(a)]
    \item\label{ratconn} the N\'{e}ron--Severi group $\NS(X_{\ol{k}})$ is torsion-free and $\H^1(X_{\ol{k}},\mathcal{O}_{X_{\ol{k}}}) = \H^2(X_{\ol{k}},\mathcal{O}_{X_{\ol{k}}}) = 0$ (this is the case if $X$ is rationally connected, for example), or
    \item\label{K3} $X$ is a K3 surface.
\end{enumerate}
\end{proposition}

\begin{proof}
In both cases~\ref{ratconn} and~\ref{K3}, $\Br(X)/\Br_0(X)$ is finite, see~\cite[Proposition~3.1.1]{Liang-Arithmetic-13} and~\cite[Theorem~1.2]{SkoroZarhinFinite}, respectively. Since the Brauer--Manin pairing factors through $\Br(X)/\Br_0(X)$, Condition~\ref{part:conndense-cond1} will follow if we can show that for each $\alpha\in \Br(X)$, the group $Z_{0,\Omega}(X)^\alpha$ is open with respect to the weak approximation topology. For this, it is enough to show the existence of $n\in\Z_{>0}$ and a finite set $S\subset \Omega_k$ such that $\mathcal{A}_{n, S}\subset Z_{0,\Omega}(X)^\alpha$ (see Section~\ref{subsec:prelim-weak-approximation} for the definition of $\mathcal{A}_{n, S}$). 

Fix $\alpha\in \Br(X)$ and let $n$ be the order of $\alpha$. By~\cite[\S 5.2]{Sko-Torsors}, there exists a finite set of places $S\subset \Omega_k$ such that $\alpha(x_v)=0$ for all $v\notin S$ and all $x_v\in (X_{k_v})_0$. Let $z=(z_v)_v\in Z_{0,\Omega}(X)$ and write $z_v=\sum_{\substack{x_v \in (X_{k_v})_0}}n_{x_v} x_v$. Then
\begin{equation}\label{eq:justS}
 \lr{ z, \alpha}_{BM}  = \sum_{v\in S}\sum_{\substack{x_v \in (X_{k_v})_0}}
  n_{x_v} \text{inv}_v(\mathrm{cores}_{k_v(x_v)/k_v}\alpha(x_v)).
\end{equation}
Now suppose that $z\in \mathcal{A}_{n, S}$, so that $z_v$ maps to $0$ in $\CH_0(X_{k_v})/n$ for all $v\in S$. Then for each $v\in S$, there exists $y_v\in Z_0(X_{k_v})$ such that $z_v$ is rationally equivalent to $ny_v$. 
Since the map 
\begin{align*}
\alpha: Z_0(X_{k_v})&\longrightarrow \Q/\Z\\
\sum_{\substack{x_v \in (X_{k_v})_0}}m_{x_v} x_v  &\longmapsto \sum_{\substack{x_v \in (X_{k_v})_0}}m_{x_v} \text{inv}_v(\mathrm{cores}_{k_v(x_v)/k_v}\alpha(x_v))
\end{align*} 
is a group homomorphism that factors via the Chow group $\CH_0(X_{k_v})$ (see~\cite{CT-Chow-93}), we have \[\alpha(z_v)=n\alpha(y_v)=0\] 
for all $v\in S$, since $\alpha$ has order $n$. Combining this with~\eqref{eq:justS} shows that $\lr{ z, \alpha}_{BM}=0$, and hence $\mathcal{A}_{n, S}\subset Z_{0,\Omega}(X)^\alpha$, as desired. Hence, Condition~\ref{part:conndense-cond1} of Theorem~\ref{thm:conndense} is satisfied. Condition~\ref{part:conndense-cond2} follows from~\cite[Proposition~3.1.1]{Liang-Arithmetic-13} and~\cite[Proof of Theorem~1.2]{Ieronymou}.
\end{proof}

\begin{remark}
It is likely that a similar proof using the Suslin homology group of degree $0$ might yield a similar result to Theorem~\ref{thm:conndense} for $X$ quasi-projective but not necessarily proper.
\end{remark}

\begin{remark}
It is probable that the strategy used to prove Theorem~\ref{thm:conndense} could be used to prove other similar results, exploiting known identities between different types of obstruction sets for rational points. For example, using the well-known equality $X(\A_L)^{\Br_1} = X(\A_L)^{\textrm{mult type}}$ where $X(\A_L)^{\textrm{mult type}}$ denotes the descent set under all linear algebraic groups of multiplicative type (see e.g.\ \cite[8.12]{Harari_Skorobogatov_2013}), one could hope to prove that $\overline{Z_{0, \Omega}(X)^{\textrm{mult type}}} = Z_{0, \Omega}(X)^{\Br_1}$ by establishing a chain of inclusions $\widetilde{Z}_{0, \Omega}(X)^{\Br_1} = \widetilde{Z}_{0, \Omega}(X)^{\textrm{mult type}} \subset Z_{0, \Omega}(X)^{\textrm{mult type}} \subset  Z_{0, \Omega}(X)^{\Br_1}$ and a result of the form $\overline{\widetilde{Z}_{0, \Omega}(X)^{\Br_1}} = Z_{0, \Omega}(X)^{\Br_1}$. 
\end{remark}

\section{A Corwin--Schlank style refined obstruction for 0-cycles}
\label{sec:corwin-schlank}

The refined obstruction sets in Definition~\ref{defn:refined-obstruction} can be generalised further to encompass analogous generalisations for rational points. We mention one such generalisation here, with the goal of extending the obstruction sets for rational points introduced in \cite{CorwinSchlank} to 0-cycles.
\begin{definition}
\label{defn:generalised-refined-obstruction}
For each finite extension $L/k$, let
\begin{equation}
\label{eq:XL-decomposition}
    X_L = \bigcup_{i=1}^{r_L} X^{(L)}_{i}
\end{equation}
denote a finite decomposition of \(X_L\) into \(L\)-schemes. For each $X^{(L)}_{i}$, define an obstruction set $\mathscr{X}^{(L)}_{i}$ with $X^{(L)}_{i}(L) \subset \mathscr{X}^{(L)}_{i} \subset X^{(L)}_{i}(L_{\Omega})$. 
Define the \emph{refined obstruction set associated to the collection $\{\{\mathscr{X}^{(L)}_{i}\}_{i=1}^{r_L}\}_{L/k}$} to be
    \[ 
\RefinedObs{X}^{\{\{\mathscr{X}^{(L)}_{i}\}_{i=1}^{r_L}\}_{L/k}} \coloneq\refrec_{X, \Omega}\left(\bigoplus_{L/k \ \textrm{finite}} \bigoplus_{i = 1}^{r_L} \Z[\mathscr{X}_i^{(L)}]\right)  \subset Z_{0,\Omega}(X),
\]
where $\refrec_{X, \Omega}\left( \bigl( (\sum_{l=1}^{t_i} n_{x_l} x_l)_{i=1}^{r_{L_j}} \bigr)_{j=1}^s \right) = \sum_{j=1}^s \sum_{i=1}^{r_{L_j}} \sum_{l=1}^{t_i} n_{x_l} N^{\Omega}_{L_j/k}(x_l)$, with $x_l \in \mathscr{X}^{(L_j)}_i$ and $n_{x_l} \in \Z$.
For $\delta\in\Z$, we define $\RefinedObsDeg{X}{\delta}^{\{\{\mathscr{X}^{(L)}_{i}\}_{i=1}^{r_L}\}_{L/k}} =\RefinedObs{X}^{\{\{\mathscr{X}^{(L)}_{i}\}_{i=1}^{r_L}\}_{L/k}}\cap Z_{0,\Omega}^\delta(X) $.
\end{definition} 

\begin{remark} Note that in \eqref{eq:XL-decomposition}, the decomposition can be anything from an open covering to a disjoint union: we impose no restrictions on the intersections of the sets \(X^{(L)}_i\). 
Furthermore, since for all $L/k$ we have $X(L) = \bigcup_{i=1}^{r_L} X_i^{(L)}(L)$, one can check that $Z_0(X) \subset \RefinedObs{X}^{\{\{\mathscr{X}^{(L)}_{i}\}_{i=1}^{r_L}\}_{L/k}}$.
\end{remark}

In this section, we  need the following additional notation. 

\begin{notation}\label{not:11}
For any non-empty (not necessarily finite) set of places 
$S \subset \Omega = \Omega_k$, we consider the projection 
$\textrm{pr}_{S}: \prod_{v \in \Omega} X(k_v) \to \prod_{v \in S} X(k_v)$. For an obstruction $\operatorname{obs}$, let $X(\A_{k, S})^{\operatorname{obs}} \coloneq \textrm{pr}_{S} \left( X(\A_{k})^{\operatorname{obs}}\right)$, $Z_{0, S}(X)\coloneq \prod_{v \in S} Z_0(X_{k_v})$, and define $\refrec_{X, S}$ similarly to $\refrec_{X, \Omega}$ (Definition~\ref{defn:refinedlocalrecombiningmap}) but restricted to the places above $v \in S$.
For any finite extension $L/k$ and any set of places $S \subset \Omega_k$, we denote by $S_L$ the set of places of $L$ above the places in $S$. We say that a triple $(\operatorname{obs}, S, k)$ where $\operatorname{obs}$ is an obstruction, $S$ is a non-empty subset of $\Omega_k$, and $k$ is a number field is an \emph{obstruction datum}. We say that a variety $V$ over $k$ satisfies \emph{very strong approximation (VSA) with respect to the obstruction datum} $(\operatorname{obs}, S, k)$ if $V(\A_{k, S})^{\operatorname{obs}} = V(k)$. 
\end{notation}

Recall the definition of finite descent obstruction sets for rational points in \eqref{eq:etobstruction}. In {\cite[Corollary 6.6 and Theorem 4.3]{CorwinSchlank}}, assuming the validity of the Section Conjecture (\cite[Conjecture 6.4]{CorwinSchlank}), Corwin and Schlank prove a very strong result on the arithmetic of rational points: for any smooth, geometrically integral variety $X$ over a number field $k$ and any non-empty (not necessarily finite) set $S$ of finite places of $k$, there is a finite open cover $X = \bigcup_{i=1}^s U_i$ such that each $U_i$ satisfies VSA with respect to the obstruction datum $(\et, S, k)$.
Since Definition~\ref{defn:generalised-refined-obstruction}
is flexible enough to incorporate obstructions in the style of Corwin--Schlank, it is natural to use it to generalise their work to the context of 0-cycles.
Assuming the validity of the Section Conjecture over all finite extensions $L/k$, we are able to prove the following analogue for 0-cycles of \cite[Corollary 6.6]{CorwinSchlank}.

\begin{theorem}\label{thm:cs-thm1}
Let $X$ be a smooth, quasi-projective, geometrically integral variety over a number field $k$. Fix a non-empty (not necessarily finite) set of finite places, \(S \subset \Omega_k\). Assume that the Section Conjecture as in \cite[Conjecture 6.4]{CorwinSchlank} holds for any finite extension $L/k$. For each finite extension \(L/k\), write $X_L = \bigcup_{i=1}^{r_L} X_i^{(L)}$ for an open cover satisfying VSA with respect to obstruction datum $(\et, S_L, L)$ guaranteed by \cite[Corollary 6.6]{CorwinSchlank}, i.e.\ for each $i = 1, ..., r_L$, $X_i^{(L)}(\A_{L, S_L})^{\et} = X_i^{(L)}(L)$.
Consider the refined obstruction set associated to the collection $\{\{X^{(L)}_{i}(\A_{L, S_L})^{\et}\}_{i=1}^{r_L}\}_{L/k}$, that is, 
\[ 
\widetilde{Z}_{0, S}(X)^{\textrm{CS-}(\et, S, k)} \coloneq\refrec_{X, S}\left(\bigoplus_{L/k \ \textrm{finite}} \bigoplus_{i = 1}^{r_L} \Z[X^{(L)}_{i}(\A_{L, S_L})^{\et}]\right)  \subset Z_{0, S}(X).
\]
Then $Z_0(X) = \widetilde{Z}_{0, S}(X)^{\textrm{CS-}(\et, S, k)}$.  Moreover, for any degree $\delta \in \Z$, we have $Z^{\delta}_0(X) = \widetilde{Z}^{\delta}_{0, S}(X)^{\textrm{CS-}(\et, S, k)}$.
\end{theorem}
\begin{proof} This is immediate since $X_i^{(L)}(\A_{L, S_L})^{\et} = X_i^{(L)}(L)$ for all $i$ and all $L$. 
\end{proof}

In fact, we can prove a slightly stronger statement, namely that starting from any finite affine open cover of $X$ over $k$, assuming the Section Conjecture holds over all $L/k$, we can refine the finite open cover in such a way that, when base-changing it to any finite extension $L/k$, the base-changed cover automatically also satisfies VSA over $L$. This means that we gain control of the open cover satisfying VSA over all $L/k$ by simply considering a carefully refined open cover over $k$.

\begin{proposition} \label{prop:onecoverforall} Let $X$ be a smooth, quasi-projective, geometrically integral variety over a number field $k$. Let $S$ be a non-empty (not necessarily finite) set of finite places of $k$. Assume that the Section Conjecture as in \cite[Conjecture 6.4]{CorwinSchlank} holds for any finite extension $L/k$. Let $X = \bigcup_{i=1}^s U_i$ be a finite affine open cover of $X$. Then there exists an effectively computable open cover $X = \bigcup_{i=1}^s \bigcup_{j=1}^{r_i} U_{i,j}$ refining $\{ U_i\}_{i=1}^s$ such that, for any finite extension $L/k$ and any $i, j$, the base change $(U_{i,j})_L$ satisfies VSA with respect to the obstruction datum $(\et, S_L, L)$.    
\end{proposition}

\begin{proof}
Fix $U \in \{ U_i \}_{i=1}^s$. Since $U$ is affine, we have a closed immersion $U \hookrightarrow \A_k^{n}$ for some $n$; moreover, we have a natural open immersion $\A_k^{n} \hookrightarrow (\P^1_k)^n$. Let \(\iota\) denote the composition
\[
\iota \colon U \hookrightarrow \A_k^{n} \hookrightarrow (\P^1_k)^n
\]
and, for \(i = 1,\ldots , n\), denote by \(\operatorname{pr}_i \colon (\P^1)^n \to \P^1\) the projection to the \(i\)-th component. 
Fix the finite open cover $\P^1_k = (\P^1_k \setminus\{\infty, 0, 1\}) \cup (\P^1_k \setminus \{2, 3, 4\})$. For any finite extension $L/k$, we also have $\P^1_L = (\P^1_L \setminus \{\infty, 0, 1\}) \cup (\P^1_L \setminus \{2, 3, 4\})$ and, since we are assuming that the Section Conjecture holds for $L/k$, by \cite[Proposition 6.5]{CorwinSchlank} both components of this open cover of $\P^1_L$ satisfy VSA with respect to the obstruction datum $(\et, S_L, L)$.  

For $i=1, ..., n$, we define $U^{(i), \P^1_k \setminus \{\infty, 0, 1\}}$ to be the pullback of the inclusion $\P^1_k \setminus \{\infty, 0, 1\} \hookrightarrow \P^1_k$ along $\operatorname{pr}_i \circ \iota$; similarly, we define $U^{(i), \P^1_k \setminus \{2,3,4\}}$ to be the pullback of the inclusion $\P^1_k \setminus \{2,3,4\} \hookrightarrow \P^1_k$ along $\operatorname{pr}_i \circ \iota$. In other words, we have the pullback diagrams in the category of $k$-schemes
\[
\begin{tikzcd}
U^{(i), \P^1_k \setminus \{\infty, 0, 1\}} \ar[r] \ar[d, hook] & \P_k^1 \setminus \{\infty, 0, 1\} \ar[d, hook]\\
	 U \ar[r, "\operatorname{pr}_i \circ \iota"] & \P^1_k
\end{tikzcd} \ \textrm{ and } \ 
\begin{tikzcd}
U^{(i), \P^1_k \setminus \{2,3,4\}} \ar[r] \ar[d, hook] & \P_k^1 
\setminus\{2,3,4\} \ar[d, hook]\\
	 U \ar[r, "\operatorname{pr}_i \circ \iota"] & \P^1_k.
\end{tikzcd}
\]

Let $\gamma \coloneq (\gamma_1, \ldots ,\gamma_n) \in \{\P^1_k \setminus \{\infty, 0, 1\}, \P^1_k \setminus \{2, 3, 4\} \}^n$.
Let $U_{\gamma} \coloneq \bigcap_{i=1}^n U^{(i), \gamma_i}$.
Then for any \(u \in U\), there is some \(\gamma \in \{\P^1_k \setminus \{\infty, 0, 1\}, \P^1_k \setminus \{2, 3, 4\} \}^n\) such that, for all $i = 1, ..., n$, we have  \(\operatorname{pr}_i \circ \iota (u) \in \gamma_i\). In particular, \(U = \bigcup_{\gamma} U_{\gamma}\) is a finite covering, where \(\gamma\) ranges over all the \(2^n\) elements of \(\{\P^1_k \setminus \{\infty, 0, 1\}, \P^1_k \setminus \{2, 3, 4\} \}^n\). 

We now claim that the collection $\left\{U_\gamma: \gamma \in \{\P^1_k \setminus \{\infty, 0, 1\}, \P^1_k \setminus \{2, 3, 4\} \}^n\right\}$ is a refinement of $U$ satisfying the required properties. More precisely, we claim that $U = \bigcup_{\gamma} U_\gamma$ is an open cover of $U$ and that for each finite extension $L/k$ and each \(\gamma\), the base change $(U_\gamma)_L$ satisfies VSA with respect to the obstruction datum $(\et, S_L, L)$. 
To see that each $U_\gamma \subset U$ is open, we just note that each of $U^{(i), \P^1_k \setminus \{\infty, 0, 1\}}$ and $U^{(i), \P^1_k \setminus \{2, 3, 4\}}$ is open in $U$, and since $U_\gamma$ is a finite intersection of opens, it is also open in $U$.
Now fix $\gamma = (\gamma_1,\ldots, \gamma_n)$ and a finite extension \(L/k\). To show that $(U_\gamma)_L$ satisfies VSA with respect to $(\et, S_L, L)$, consider the product $Y_\gamma \coloneq \gamma_1 \times_k \ldots \times_k \gamma_n \subset (\P^1_k)^n$. We note that, since each $U^{(i), \gamma_i}$ is the pullback of $\gamma_i \hookrightarrow \P^1_k$ along $\operatorname{pr}_i \circ \iota$, by construction $U_\gamma$ is the pullback of $Y_\gamma \hookrightarrow (\P_k^1)^n$ along $\iota$, that is, we have the pullback diagram
\[
\begin{tikzcd}
U_\gamma \ar[r, hook, "\iota"] \ar[d, hook] & Y_\gamma \ar[d, hook]\\
	 U \ar[r, hook, "\iota"] & (\P^1_k)^n.
\end{tikzcd} 
\]

We claim that $\iota(U_\gamma)$ is locally closed in $Y_\gamma$. To see this, note that $\iota: U \hookrightarrow (\P^1_k)^n$ is the composition of a closed immersion followed by an open immersion, both of which are stable under pullbacks \cite[Tag 01JY]{stacks-project}. It follows that the pullback map $\iota: U_\gamma \hookrightarrow Y_\gamma$ is also the composition of a closed immersion followed by an open immersion, and thus, by definition, is a locally closed immersion~\cite[Tag 01IM]{stacks-project}. This proves our claim. 

All these constructions are compatible with base change to $L$, so $\iota_L(U_\gamma)_L$ is locally closed in $(Y_\gamma)_L = (\gamma_1)_L \times_L \ldots \times_L (\gamma_n)_L$.
By assumption, the Section Conjecture holds over $L$ and therefore by \cite[Proposition 6.5]{CorwinSchlank} each $(\gamma_i)_L$ satisfies VSA with respect to the obstruction datum $(\et, S_L, L)$. By \cite[Lemma 2.5(ii)]{CorwinSchlank}, VSA is preserved under products, so $(Y_\gamma)_L$ satisfies VSA with respect to $(\et, S_L, L)$. Finally by \cite[Lemma 2.5(i)]{CorwinSchlank}, $(U_\gamma)_L$ also  satisfies VSA with respect to $(\et, S_L, L)$, since it is locally closed in $(Y_\gamma)_L$. Putting everything together, by considering for each $U$ in the open cover for $X$ its refinement $\{U_\gamma\}_\gamma$ constructed above, we have found a refined open cover of $X$ where each component base-changed to any finite extension $L/k$ satisfies  VSA with respect to the obstruction datum $(\et, S_L, L)$, as required. \end{proof}

\begin{remark} In fact, in the proof of Proposition \ref{prop:onecoverforall}, the only place where the Section Conjecture over all finite extensions $L/k$ is used is to guarantee that $\P^1_L \setminus \{\infty, 0, 1\}$ and $\P^1_L \setminus \{2, 3, 4\}$ both satisfy VSA with respect to the obstruction datum $(\et, S_L, L)$. Instead of requiring the full Section Conjecture to hold, we thus could just require that $\P^1_L \setminus \{\infty, 0, 1\}$ and $\P^1_L \setminus \{2, 3, 4\}$ both satisfy VSA with respect to the obstruction datum $(\et, S_L, L)$ for all finite extensions $L/k$, which is a much weaker requirement.
\end{remark}

The following corollary is an immediate consequence of Theorem~\ref{thm:cs-thm1} and Proposition~\ref{prop:onecoverforall}.

\begin{corollary}\label{cor:corwin-schlank-effective-refinement}
Let $X$ be a smooth, quasi-projective, geometrically integral variety over a number field $k$. Fix a non-empty (not necessarily finite) set of finite places \(S \subset \Omega_k\). Assume that the Section Conjecture as in \cite[Conjecture 6.4]{CorwinSchlank} holds for any finite extension $L/k$. Let $X = \bigcup_{i=1}^s U_i$ be a finite open affine cover, and let $X = \bigcup_{i=1}^s \bigcup_{j=1}^{r_i} U_{i,j}$ be the refinement of  $\bigcup_{i=1}^s U_i$ constructed in Proposition~\ref{prop:onecoverforall}. Consider the refined obstruction set for the collection $\left\{\left\{U_{i,j}(\A_{L, S_L})^{\et}\right\}_{\substack{i=1, ..., s\\ j = 1, ..., r_i}}\right\}_{L/k}$, that is, 
\[ 
\widetilde{Z}_{0, S}(X)^{\textrm{CS-}(\et, S, k)} \coloneq\refrec_{X, S}\left(\bigoplus_{L/k \ \textrm{finite}} \bigoplus_{i = 1}^{s}\bigoplus_{j=1}^{r_i} \Z[U_{i,j}(\A_{L, S_L})^{\et}]\right)  \subset Z_{0, S}(X).
\]
Then $Z_0(X) = \widetilde{Z}_{0, S}(X)^{\textrm{CS-}(\et, S, k)}$.  Moreover, for any degree $\delta \in \Z$, we have $Z^{\delta}_0(X) = \widetilde{Z}^{\delta}_{0, S}(X)^{\textrm{CS-}(\et, S, k)}$.
\end{corollary}

\bibliographystyle{alpha}
\bibliography{refs}

\end{document}